\documentclass[10pt,leqno]{article}

\usepackage{amsmath,amssymb,amsthm,mathrsfs,dsfont}

\usepackage[margin=3cm]{geometry} 

\usepackage{titlesec,hyperref}

\usepackage{color}

\usepackage{fancyhdr}
\titleformat{\subsection}{\it}{\thesubsection.\enspace}{1pt}{}

\newtheorem{theo}{Theorem}[section]
\newtheorem{lemm}[theo]{Lemma}

\newtheorem{prop}[theo]{Proposition}
\newtheorem{rema}[theo]{Remark}
\numberwithin{equation}{section}

\allowdisplaybreaks 

\newcommand\lm{{\lesssim}}

\begin{document}
\title{Global strong solutions and large time behavior to a micro-macro model for compressible polymeric fluids near equilibrium
\hspace{-4mm}
}

\author{ Wenjie $\mbox{Deng}^1$ \footnote{email: detective2028@qq.com},\quad
Wei $\mbox{Luo}^1$\footnote{E-mail:  luowei23@mail2.sysu.edu.cn} \quad and\quad
 Zhaoyang $\mbox{Yin}^{1,2}$\footnote{E-mail: mcsyzy@mail.sysu.edu.cn}\\
 $^1\mbox{Department}$ of Mathematics,
Sun Yat-sen University, Guangzhou 510275, China\\
$^2\mbox{Faculty}$ of Information Technology,\\ Macau University of Science and Technology, Macau, China}

\date{}
\maketitle
\hrule

\begin{abstract}
In this paper, we mainly study the global strong solutions and its long time decay rates of all order spatial derivatives to a micro-macro model for compressible polymeric fluids with small initial data. This model is a coupling of isentropic compressible Navier-Stokes equations with a nonlinear Fokker-Planck equation. We first prove that the micro-macro model admits a unique global strong solution provided the initial data are close to equilibrium state for $d\geq 2$. Moreover, for $d\geq 3$, we also show a new critical Fourier estimation that allow us to give the long time decay rates of $L^2$ norm for all order spatial derivatives.  \\
\vspace*{5pt}

\vspace*{5pt}
\end{abstract}

\vspace*{10pt}

\tableofcontents

\section{Introduction}
   In this paper, we consider a micro-macro model for compressible polymeric fluids near equilibrium with dimension $d\geq2$ :
   \begin{align}\label{eq0}
\left\{
\begin{array}{ll}
\varrho_t+{\rm div}(\varrho u)=0 , \\[1ex]
(\varrho u)_t+{\rm div}(\varrho u\otimes u)-\rm div\Sigma{(u)}+\frac 1 {Ma^2} \nabla{P(\varrho)}=\frac 1 {De} \frac {\kappa} {r} \rm div~\tau, \\[1ex]
\psi_t+u\cdot\nabla\psi={\rm div}_{q}[- \nabla u \cdot{q}\psi+\frac {\sigma} {De}\nabla_{q}\psi+\frac {1} {De\cdot r}\nabla_{q}\mathcal{U}\psi],  \\[1ex]
\tau_{ij}=\int_{\mathbb{R}^{d}}(q_{i}\nabla_{q_j}\mathcal{U})\psi dq, \\[1ex]
\varrho|_{t=0}=\varrho_0,~~u|_{t=0}=u_0,~~\psi|_{t=0}=\psi_0, \\[1ex]
\end{array}
\right.
\end{align}
where $\varrho(t,x)$ is the density of the solvent, $u(t,x)$ stands for the velocity of the polymeric liquid and $\psi(t,x,q)$ denotes the distribution function for the internal configuration. Here the polymer elongation $q\in\mathbb{R}^{d}$ , $x\in\mathbb{R}^{d}$ and $t\in[0,\infty)$.
The notation $\Sigma{(u)}=\mu(\nabla u+\nabla^{T} u)+\mu'{\rm div}u\cdot Id$ stands for the stress tensor, with $\mu$ and $\mu'$ being the viscosity coefficients satisfying the relation $\mu>0$ and $2\mu+\mu'>0$.
The pressure obeys the so-called $\gamma$-law: $P(\varrho)=a\varrho^\gamma$ with $\gamma\geq1, a>0$. $\sigma$ is a constant satisfing the relation $\sigma=k_BT_a$, where $k_B$ is the Boltzmann constant and $T_a$ is the absolute temperature.  Furthermore, $r>0$ is related to
the linear damping mechanism in dynamics of the microscopic variable $q$, and $\kappa > 0$ is some parameter describing the ratio between kinetic and elastic energy. The parameter $De$ denotes the Deborah number, which represents the ratio of the time scales for elastic stress relaxation, so it characterizes the fluidity of the system. The Mach number $Ma$ describes the ratio between the fluid velocity and the sound speed, which measures the compressibility of the system.
  Moreover, the potential $\mathcal{U}(q)$ follows the same assumptions as that of \cite{2017Global} :
\begin{align}\label{potential1}
\left\{
\begin{array}{ll}
|q|\lesssim(1+|\nabla_q\mathcal{U}|),\\
\Delta_q\mathcal{U}\leq C + \delta|\nabla_q\mathcal{U}|^2,\\
\int_{\mathbb{R}^{d}}|\nabla_q\mathcal{U}|^2\psi_{\infty}dq\leq C,~~\int_{\mathbb{R}^{d}}|q|^4\psi_{\infty}dq\leq C,
\end{array}
\right.
\end{align}
with $\delta\in(0,1)$ and
\begin{align}\label{potential2}
\left\{
\begin{array}{ll}
|\nabla^k_q(q\nabla_q\mathcal{U})|\lesssim(1+|q||\nabla_q\mathcal{U}|),\\
\int_{\mathbb{R}^d}|\nabla^k_q(q\nabla_q\mathcal{U}\sqrt{\psi_{\infty}})|^2dq\leq C,\\
|\nabla^k_q(\Delta_q\mathcal{U}-\frac{1}{2}|\nabla_q\mathcal{U}|^2)|\lesssim(1+|\nabla_q\mathcal{U}|^2),~~~~~
\end{array}
\right.
\end{align}
with the integer $k\in[1,3]$.  The most simplest case among them is the Hookean spring $\mathcal{U}(q)=\frac{1}{2}|q|^2$. From the Fokker-Planck equation, one can derive the following compressible Oldroyd-B equation, which has been studied in depth in \cite{incomprelimDFang,incomprelimZLei,compre1,compre2,2018Global} :
\begin{align}\label{tauequ}
\tau_t + u\cdot\nabla\tau + 2\tau  + Q(\nabla u,\tau) = Du~,~~~~~~Du^{ij}=\mathop{\sum}\limits_{l,m=1}^d\nabla^lu^m\int_{\mathbb{R}^d}q^lq^mq^iq^j e^{-|q|^2}dq~.
\end{align}
It is universally  known that the system \eqref{eq0} can be used to described the fluids coupling polymers. The system is of great interest in many branches of physics, chemistry, and biology, see \cite{Bird1977,Doi1988}. In this model, a polymer is idealized as an ``elastic dumbbell" consisting of two ``beads" joined by a spring that can be modeled by a vector $q$. The polymer particles are described by a probability function $\psi(t,x,q)$ satisfying that $\int_{\mathbb{R}^{d}}
\psi(t,x,q)dq =1$, which represents the distribution of particles' elongation vector $q\in \mathbb{R}^{d}$. At the level of liquid, the system couples the Navier-Stokes equation for the fluid velocity with a Fokker-Planck equation describing the evolution of the polymer density. One can refer to  \cite{Bird1977,Doi1988,Masmoudi2008,Masmoudi2013} for more details.

%
%
%


In this paper we will take $a,~\sigma,~\kappa,~r,~De,~Ma$ equal to $1$.
It is easy to check that $(1,0,\psi_{\infty})$ with $$\psi_{\infty}(q)=\frac{e^{-\mathcal{U}(q)}}{\int_{\mathbb{R}^{d}}e^{-\mathcal{U}(q)}dq}~,~~~~~~~~$$
is  a stationary solution to the system \eqref{eq0}. Considering the perturbations near the global equilibrium:
\begin{align*}
\rho=\varrho-1,~~u=u~~~\text{and}~~~g=\frac {\psi-\psi_\infty} {\psi_\infty}~,~~~~~
\end{align*}
then we can rewrite \eqref{eq0} as the following system :
\begin{align}\label{eq1}
	~~~~~~~~~~~~~~~~~~
\left\{
\begin{array}{ll}
\rho_t+\rm divu(1+\rho)=-u\cdot\nabla\rho , \\[1ex]
u_t-\frac 1 {1+\rho} \rm div\Sigma{(u)}+\frac {P'(1+\rho)} {1+\rho} \nabla\rho=-u\cdot\nabla u+\frac 1 {1+\rho} \rm div\tau, \\[1ex]
g_t+\mathcal{L}g=-u\cdot\nabla g-\frac 1 {\psi_\infty}\nabla_q\cdot(\nabla uqg\psi_\infty)-\rm divu-\nabla u q\nabla_{q}\mathcal{U},  \\[1ex]
\tau_{ij}(g)=\int_{\mathbb{R}^{d}}(q_{i}\nabla_{q_j}\mathcal{U})g\psi_\infty dq, \\[1ex]
\rho|_{t=0}=\rho_0,~~u|_{t=0}=u_0,~~g|_{t=0}=g_0, \\[1ex]
\end{array}
\right.
\end{align}
where $\mathcal{L}g=-\frac 1 {\psi_\infty}\nabla_q\cdot(\psi_\infty\nabla_{q}g)$ .

\subsection{Short reviews for the incompressible polymeric fluid models}
Before going into further introduction for the compressible model, we would like to review some well-known results on the mathematical analysis for the incompressible case firstly.
In recent decades, the incompresssible polymeric fluid models have been extensitively studied, since its importance and challenging. M. Renardy \cite{Renardy} established the local well-posedness in Sobolev spaces with potential $\mathcal{U}(q)=(1-|q|^2)^{1-\sigma}$ for $\sigma>1$. Later, B. Jourdain, T. Leli\`{e}vre, and C. Le Bris \cite{Jourdain} proved local existence of a stochastic differential equation with potential $\mathcal{U}(q)=-k\log(1-|q|^{2})$ in the case $k>3$ for a Couette flow. For the co-rotation case, P. L. Lions, N. Masmoudi \cite{Lions-Masmoudi} constructed global weak solution for some Oldroyd models. J. Y. Chemin and N. Masmoudi \cite{lifespan} showed a sufficient condition of non-breakdown for an incompressible viscoelastic fluid of the Oldroyd type by establishing a new priori estimate for 2D Navier-Stokes system
and a losing estimate for the transport equation. We point out that these estimates are of significance in the proof of the strong solutions for viscoelastic fluids. N. Masmoudi, P. Zhang and Z. Zhang \cite{Masmoudi2dfluid} obtained global well-posedness for 2D polymeric fluid models under the co-rotational assumption without any small conditions. In addition, Z. Lei , N. Masmoudi and Y. Zhou \cite{Z.Leinewmethod} provided a new method to prove and improve the criterion for viscoelastic systems of Oldroyd type obtained in \cite{lifespan}. It is worth noting that this new method is much easier and can be widely applied to address other problems involving the losing a prior estimate.
\subsection{Short reviews for the compressible polymeric fluid models}
Significantly, the compressible polymeric fluid models are more physical. When taking density into account,  the models become more difficult and complex to deal with. Z. Lei \cite{incomprelimZLei} first investigated the incompressible limit problem of the compressible Oldroyd-B model in
torus. Recently, D. Fang and R. Zi \cite{incomprelimDFang} studied the global well-posedness for compressible Oldroyd-B
model in critical Besov spaces with $d \geq 2$ . In \cite{2018Global}, Z. Zhou, C. Zhu and R. Zi proved the global well-posedness and decay rates for the 3-D compressible Oldroyd-B model. The compressible Oldroyd-B type model based on the deformation tensor can be found in \cite{compre1,compre2}. Recently, N. Jiang, Y. Liu and T. Zhang \cite{2017Global} employed the energetic variational method to derive a micro-macro model for compressible polymeric fluids and proved the global existence near the global equilibrium with $d=3$.
\subsection{Main results}

The long time behavior for polymeric models \cite{arma06} is first investigated by Jourdain et al. L. He and P. Zhang \cite{He2009} studied the long time decay of the $L^2$ norm to the incompressible Hooke dumbell models and founded that the solutions tends to the equilibrium by $(1+t)^{-\frac{3}{4}}$ when the initial perturbation is additionally bounded in $L^1$. Recently, M. Schonbek \cite{Schonbekdecay} studied the $L^2$ decay of the velocity $u$ to the co-rotation FENE dumbbell model and proved that velocity $u$ tends to zero in $L^2$ by  $(1+t)^{-\frac{d}{4}+\frac{1}{2}}$, $d\geq 2$ with the assumption that $u_0\in L^1$. More recently, W. Luo and Z. Yin \cite{Luo-Yin,Luo-Yin2} improved Schonbek's result and showed that the optimal decay rate of velocity $u$ in $L^2$ should be $(1+t)^{-\frac{d}{4}}$. Moreover, the optimal decay rates of the higher-order spatial derivatives of the two-phase fluid model were obtianed in \cite{2020Wuguochun}.

To our best knowledge, large time behaviour for the compressible polymeric fluid model with potential $\mathcal{U}(q)$ satisfying conditions \eqref{potential1} and \eqref{potential2} has not been studied yet. In this paper, we firstly study global solutions of system $(\ref{eq1})$ under initial data near equilibrium in $H^s$ with $s>1+\frac d 2$ and $d\geq2$. By virtue of Fourier splitting method, we then derive the optimal time decay for all order spatial derivatives of $(\rho,u)$ for $d\geq3$. 

Our main results can be stated as follows.
\begin{theo}\label{th1}
Let $d\geq 2~and~s>1+\frac d 2$. Let $(\rho,u,g)$ be a classical solution of system \eqref{eq1} with the initial data $(\rho_0,u_0,g_0)$ satisfying the conditions $\int_{\mathbb{R}^d} g_{0}\psi_{\infty}dq=0$ and $1+g_0>0$. Then, there exists some sufficiently small constant $\epsilon_0$ such that if
	\begin{align*}
	E_{\lambda}(0)=\|\rho_0\|^2_{H^s}+\|u_0\|^2_{H^s}+\|g_0\|^2_{H^s(\mathcal{L}^2)} + \lambda\|\langle q\rangle g_0\|^2_{H^{s-1}(\mathcal{L}^2)}\leq \epsilon_0~,
\end{align*}
then the compressible system \eqref{eq1} admits a unique global classical solution $(\rho,u,g)$ with $\int_{\mathbb{R}^d} g\psi_{\infty}dR=0$ and $1+g>0$ , and we have
	\begin{align*}
	\sup_{t\in[0,+\infty)} E_{\lambda}(t)+\int_{0}^{\infty}D_{\lambda}(t)dt\leq \epsilon~,
\end{align*}
where $\epsilon$ is a small constant dependent on the viscosity coefficients.
\end{theo}
\begin{theo}\label{th2}
Let $d\geq 3$. Let $(\rho,u,g)$ be a global strong solution of system \eqref{eq1} with the initial data $(\rho_0,u_0,g_0)$ under the conditions in Theorem \ref{th1}~. In addition, if 
$(\rho_0,u_0)\in \dot{B}^{-\frac{d}{2}}_{2,\infty}\times \dot{B}^{-\frac{d}{2}}_{2,\infty}$ and $g_0\in \dot{B}^{-\frac{d}{2}}_{2,\infty}(\mathcal{L}^2)$, 
then there exists a constant $C$ such that
	\begin{align*}
		~~~~~~~~
	\left\{
	\begin{array}{ll}
		\|\Lambda^\sigma(\rho,u)\|_{L^2} \leq C(1+t)^{-\frac{d}{4}-\frac{\sigma}{2}} ~,~~~~~~\sigma\in[0,s]~, \\
		\|\Lambda^\nu g\|_{L^2(\mathcal{L}^2)} ~\leq C(1+t)^{-\frac{d}{4}-\frac{\nu}{2}-\frac{1}{2}} ~,~~~\nu\in[0,s-1]~,\\
		\|\Lambda^\gamma g\|_{L^2(\mathcal{L}^2)} ~\leq C(1+t)^{-\frac{d}{4}-\frac{s}{2}}~,~~~~~~\gamma\in[s-1,s]~.
	\end{array}
	\right.
\end{align*}
\end{theo}
We would like to introduce the difficulties and our main ideas to overcome them.
\begin{rema}
 
Different from the classical potential $\mathcal{U}=\frac{1}{2}|q|^2$, the lack of  $\|(\rho,u)\|_{L^1}$ leads the loss of $\frac{d}{8}$ time decay rate (One can see the estimate of term $B_2$ in Section 3 below for more details) , which forces us to obtain the weaker  decay rate in $L^2$ at the very beginning. To overcome this difficulty, we prove the global solutions of system $(\ref{eq1})$ considered in Theorem \ref{th1} belong to some Besov space with negative index as that of \cite{luozhaonanFENE}. In reward, it helps to improve the time decay rate in $L^2$ obtained by using the bootstrap argument. 
\end{rema}
\begin{rema}
It is worthy mentioning that the optimal decay rates of $(\rho,u,g)$ in $\dot H^s$ is absolutely innovative and it fails to obtain the decay rate of $(\rho,u)$ in $\dot{H}^s$ by the same way as in $L^2$ since
$$\|\Lambda^s(\rho,u)\|^2_{L^2} + \eta\langle\Lambda^{s-1}u,\Lambda^s\rho\rangle \simeq \|\Lambda^s(\rho,u)\|^2_{L^2}$$
does not hold anymore for any $\eta>0$. To overcome such obstacle, motivated by \cite{2020Wuguochun}, a critical Fourier splitting estimate is established in Lemma \ref{decayinH.s} (see Section 5 below) , which implies that the dissipation of $\rho$ only in high frequency fully enables us to derive optimal decay rate in $\dot{H}^s$. It should be pointed out that different from the frequency decomposition carrying out in \cite{2020Wuguochun}, the frequency decomposition considered in this paper is in time dependency, which conduces to obtain optimal decay rate without finite induction argument. 
\end{rema}
\begin{rema}
	To our best knowledge, it is the first result for the time decay rate in $\dot{H}^s$ of global solutions to the compressible Oldroyd-B model. Moreover, referring to \cite{He2009}, we obtain the faster time decay rate for $\tau$ and $g$. Indeed, we gain more $\frac{1}{2}$ decay rate for $(\rho,u)$ in $\dot H^{\sigma}$, $\sigma\in(s-1,s]$. In reward, we gain $\frac{1}{2}$ faster decay rate for $\tau$ $\left(\text{or}~ g\right)$ in $\dot H^{\sigma}$ $\left(\text{in}~ \dot H^{\sigma}(\mathcal{L}^2)\right)$ with $\sigma\in(s-2,s]$ than we can obtain originally.
\end{rema}
The paper is organized as follows. In Section 2, we introduce some notations and  give some preliminaries which will be used in the sequel. In Section 3, we prove that the micro-macro model for compressible polymeric fluids admits a unique global strong solution provided the initial data are close to equilibrium state for $d\geq2$. In Section 4 and Section 5, we study the long time decay rate in $L^2$ and $\dot H^s$ respectively for the solutions established in Section 3 by using Fourier splitting method and end up with the proof of Theorem \ref{th2} by interpolation.

\section{Preliminaries}
  In this section we will introduce some notations and useful lemmas which will be used in the sequel.

 If the function spaces are over $\mathbb{R}^d$ for the variable $x$ and $q$, for simplicity, we drop $\mathbb{R}^d$ in the notation of function spaces if there is no ambiguity.

For $p\geq1$, we denote by $\mathcal{L}^{p}$ the space
$$\mathcal{L}^{p}=\big\{f \big|\|f\|^{p}_{\mathcal{L}^{p}}=\int_{\mathbb{R}^d} \psi_{\infty}|f|^{p}dq<\infty\big\}.~~~~~~~~~~~~~~~~$$

We will use the notation $L^{p}_{x}(\mathcal{L}^{q})$ to denote $L^{p}[\mathbb{R}^{d};\mathcal{L}^{q}]:$
$$L^{p}_{x}(\mathcal{L}^{q})=\big\{f \big|\|f\|_{L^{p}_{x}(\mathcal{L}^{q})}=(\int_{\mathbb{R}^{d}}(\int_{\mathbb{R}^d} \psi_{\infty}|f|^{q}dq)^{\frac{p}{q}}dx)^{\frac{1}{p}}<\infty\big\}.$$

The symbol $\widehat{f}=\mathcal{F}(f)$ denotes the Fourier transform of $f$.
Let $\Lambda^s f=\mathcal{F}^{-1}(|\xi|^s \widehat{f})$.
If $s\geq0$, we can denote by $H^{s}(\mathcal{L}^{2})$ the space
$$~~~~~~H^{s}(\mathcal{L}^{2})=\{f\big| \|f\|^2_{H^{s}(\mathcal{L}^{2})}=\int_{\mathbb{R}^{d}}\int_{\mathbb{R}^{d}}(|f|^2+|\Lambda^s f|^2)\psi_\infty dRdx<\infty\}.$$
Then we introduce the energy and energy dissipation functionals for the fluctuation $(\rho,u,g)$ as follows:
\begin{align*}
	~~~~~~~~~~~~~~~E_{\lambda}(t)&=\sum_{m=0,s}\left(\|h(\rho)^{\frac 1 2}\Lambda^m\rho \|^2_{L^2}+\|(1+\rho)^{\frac 1 2}\Lambda^m u\|^2_{L^2}+\|\Lambda^mg\|^2_{L^2(\mathcal{L}^{2})} \right) \\ \notag
	& ~~~~ + \sum_{m=0,s-1}\lambda\|\langle q \rangle \Lambda^m g\|^2_{L^2(\mathcal{L}^{2})}~,
\end{align*}
and
\begin{align*}
	~~~~~~~~~~~~~~~~D_{\lambda}(t)&=\gamma\|\nabla\rho\|^2_{H^{s-1}}+\mu\|\nabla u\|^2_{H^{s}}+(\mu+\mu')\|{\rm div}u\|^2_{H^{s}}+\|\nabla_q g\|^2_{H^{s}(\mathcal{L}^{2})} \\ \notag
	&~~~+ \lambda\|\langle q \rangle \nabla_q g\|^2_{H^{s-1}(\mathcal{L}^{2})}~,~~~~~~~~~~~~~~~~~~~~~~~~~~~~~
\end{align*}
where positive constant $\lambda$ is small enough. Sometimes we write $f\lm g$ instead of $f\leq Cg$, where $C$ is a constant. We agree that $\nabla$ stands for $\nabla_x$ and $div$ stands for $div_x$.\\

We recall the Littlewood-Paley decomposition theory and and Besov spaces.
\begin{prop}\cite{Bahouri2011}\label{prop0}
Let $\mathcal{C}$ be the annulus $\{\xi\in\mathbb{R}^d:\frac 3 4\leq|\xi|\leq\frac 8 3\}$. There exist radial functions $\chi$ and $\varphi$, valued in the interval $[0,1]$, belonging respectively to $\mathcal{D}(B(0,\frac 4 3))$ and $\mathcal{D}(\mathcal{C})$, and such that
$$ \forall\xi\in\mathbb{R}^d,\ \chi(\xi)+\sum_{j\geq 0}\varphi(2^{-j}\xi)=1, $$
$$ \forall\xi\in\mathbb{R}^d\backslash\{0\},\ \sum_{j\in\mathbb{Z}}\varphi(2^{-j}\xi)=1, ~~~$$
$$ |j-j'|\geq 2\Rightarrow\mathrm{Supp}\ \varphi(2^{-j}\cdot)\cap \mathrm{Supp}\ \varphi(2^{-j'}\cdot)=\emptyset, $$
$$ ~~j\geq 1\Rightarrow\mathrm{Supp}\ \chi(\cdot)\cap \mathrm{Supp}\ \varphi(2^{-j}\cdot)=\emptyset. $$
The set $\widetilde{\mathcal{C}}=B(0,\frac 2 3)+\mathcal{C}$ is an annulus, then
$$ |j-j'|\geq 5\Rightarrow 2^{j}\mathcal{C}\cap 2^{j'}\widetilde{\mathcal{C}}=\emptyset. ~~$$
Further, we have
$$ \forall\xi\in\mathbb{R}^d,\ \frac 1 2\leq\chi^2(\xi)+\sum_{j\geq 0}\varphi^2(2^{-j}\xi)\leq 1, $$
$$ \forall\xi\in\mathbb{R}^d\backslash\{0\},\ \frac 1 2\leq\sum_{j\in\mathbb{Z}}\varphi^2(2^{-j}\xi)\leq 1. ~~~~~$$
\end{prop}

$\mathcal{F}$ represents the Fourier transform and  its inverse is denoted by $\mathcal{F}^{-1}$.
Let $u$ be a tempered distribution in $\mathcal{S}'(\mathbb{R}^d)$. For all $j\in\mathbb{Z}$ , define
$$~~~~~~
\dot\Delta_j u=\mathcal{F}^{-1}\left(\varphi(2^{-j}\cdot)\mathcal{F}u\right)~~~~~~and~~~~~~
\dot S_j u=\sum_{j'<j}\dot\Delta_{j'}u~.~~~~~~~~~~~~~~~
$$
Then the Littlewood-Paley decomposition is given as follows:
$$ u=\sum_{j\in\mathbb{Z}}\dot\Delta_j u \quad \text{in} \ \ \ \mathcal{S}^{'}_h(\mathbb{R}^d). ~~~~~$$
Let $s\in\mathbb{R},\ 1\leq p,r\leq\infty.$ The homogeneous Besov space $\dot B^s_{p,r}$ and $\dot B^s_{p,r}(\mathcal{L}^{p'})$ is defined by
$$ \dot B^s_{p,r}=\{u\in S^{'}_h:\|u\|_{\dot B^s_{p,r}}=\Big\|\left(2^{js}\|\dot\Delta_j u\|_{L^p}\right)_j \Big\|_{l^r(\mathbb{Z})}<\infty\}, ~~~~$$
$$ \dot B^s_{p,r}(\mathcal{L}^{p'})=\{\phi\in S^{'}_h:\|\phi\|_{\dot B^s_{p,r}(\mathcal{L}^{p'})}=\Big\|\left(2^{js}\|\dot\Delta_j \phi\|_{L_{x}^{p}(\mathcal{L}^{p'})}\right)_j \Big\|_{l^r(\mathbb{Z})}<\infty\}.$$

The following lemma is about the embedding between Lesbesgue and Besov spaces.
\begin{lemm}\cite{Bahouri2011}\label{lemma1}
	Let $1\leq p\leq 2$ and $d\geq3$. Then it holds that 
	$$L^{p}\hookrightarrow \dot{B}^{\frac{d}{2}-\frac d p}_{2,\infty}.$$
	Moreover, it follows that
	$$L^2=\dot{B}^0_{2,2} .~~~$$
\end{lemm}

The following lemma is the Gagliardo-Nirenberg inequality.
\begin{lemm}\cite{Nirenberg}\label{Lemma3}
Let $d\geq2,~p\in[2,+\infty)$ and $0\leq s,s_1\leq s_2$, then there exists a constant $C$ such that
 $$\|\Lambda^{s}f\|_{L^{p}}\leq C \|\Lambda^{s_1}f\|^{1-\theta}_{L^{2}}\|\Lambda^{s_2} f\|^{\theta}_{L^{2}},$$
where $0\leq\theta\leq1$ and $\theta$ satisfy
$$ s+d(\frac 1 2 -\frac 1 p)=s_1 (1-\theta)+\theta s_2.$$
Note that we require that $0<\theta<1$, $0\leq s_1\leq s$, when $p=\infty$.
\end{lemm}

The following lemma allows us to estimate the extra stress tensor $\tau$.
\begin{lemm}\cite{2017Global}\label{Lemma4}
 Assume $g\in H^s(\mathcal{L}^{2})$ with $\int_{\mathbb{R}^{d}} g\psi_\infty dq=0$, then it follows that
\begin{align*}
\left\{
\begin{array}{ll}
\|\nabla_q\mathcal{U}g\|_{\dot{H}^\sigma(\mathcal{L}^{2})},\|qg\|_{\dot{H}^\sigma(\mathcal{L}^{2})}\lesssim\|\nabla_qg\|_{\dot{H}^\sigma(\mathcal{L}^{2})},\\
\|q\nabla_q\mathcal{U}g\|_{\dot{H}^\sigma(\mathcal{L}^{2})},\||q|^2g\|_{{\dot{H}^\sigma(\mathcal{L}^{2})}}\lesssim\|\langle q\rangle\nabla_qg\|_{\dot{H}^\sigma(\mathcal{L}^{2})},
\end{array}
\right.
\end{align*}
for any $\sigma\in[0,s]$.
\end{lemm}

\begin{lemm}\cite{2017Global}\label{Lemma5}
  Assume $g\in H^s(\mathcal{L}^{2})$, then it holds that
$$|\tau(g)|^2\lesssim \|g\|^2_{\mathcal{L}^{2}}\lesssim\|\nabla_qg\|^2_{\mathcal{L}^{2}} .$$
\end{lemm}

\begin{lemm}\cite{Moser1966A}\label{Lemma6}
Let $s\geq 1$. Assume $p_1,...,p_4$ and $p\in (1,\infty)$ with $\frac 1 p =\frac 1 {p_1}+\frac 1 {p_2}=\frac 1 {p_3}+\frac 1 {p_4}$. Then it holds that
$$\|[\Lambda^s, f]g\|_{L^p}\leq C\left(\|\Lambda^{s}f\|_{L^{p_1}}\|g\|_{L^{p_2}}+\|\nabla f\|_{L^{p_3}}\|\Lambda^{s-1}g\|_{L^{p_4}}\right).~~$$
Analogously,
$$\|[\Lambda^s, f]g\|_{L^2(\mathcal{L}^{2})}\leq C\left(\|\Lambda^{s}f\|_{L^2}\|g\|_{L^\infty(\mathcal{L}^{2})}+\|\nabla f\|_{L^\infty}\|\Lambda^{s-1}g\|_{L^2(\mathcal{L}^{2})}\right).$$
\end{lemm}

\section{Global strong solutions with small data}
This section is devoted to investigating the global classical solutions for the micro-macro model for compressible polymeric fluids with dimension $d\geq2$ . We divide it into two Propositions to prove Theorem \ref{th1} . First, using a standard iterating method as in \cite{2017Global} , we can deduce that the existence of local solutions in some appropriate spaces. We omit the proof and give the following Proposition.
\begin{prop}\label{pro1}
Let $d\geq 2~and~s>1+\frac d 2$. Assume $E_{\lambda}(0)\leq \frac {\epsilon} 2$. Then there exist a time $T>0$ such that the micro-macro polymeric system \eqref{eq1} admits a unique local classical solution $(\rho_,u,g)\in L^{\infty}(0,T;H^s\times H^s\times H^s(\mathcal{L}^2))$ and we have
\begin{align*}
\sup_{t\in[0,T]} E_{\lambda}(t)+\int_{0}^{T}D_{\lambda}(t)dt\leq \epsilon,
\end{align*}
\end{prop}
Then we prove a key global priori estimate for local solutions in the following proposition.
\begin{prop}\label{pro2}
Let $d\geq 2~and~s>1+\frac d 2$. Let $(\rho,u,g)\in L^{\infty}(0,T;H^s\times H^s\times H^s(\mathcal{L}^2))$  be local solutions constructed in Proposition \ref{pro1} . If $\mathop{\sup}\limits_{t\in[0,T)} E_{\lambda}(t)\leq \epsilon$, then there exist $C_0>1$ such that
\begin{align}\label{decayenergy}
\sup_{t\in[0,T]} E_{\lambda}(t)+\int_{0}^{T}D_{\lambda}(t)dt\leq C_{0}E_{\lambda}(0)~.
\end{align}
\end{prop}
We divide it into three steps to prove Proposition \ref{pro2} . The estimates in the first two steps is similar to those in \cite{luozhaonanFENE} with different external force $\tau$. It should be underlined that such difference results in additional estimates for derivatives with micro weight as shown in the third step.
\subsection{Estimates on lower order derivatives}
\begin{lemm}\label{low}
Let $(\rho_,u,g)$ be local classical solutions considered in proposition \ref{pro2} , then there exists positive constant $\eta$ and $\delta$ such that
\begin{align}\label{low estimate}
		\frac {d} {dt}& \left(\|h(\rho)^{\frac 1 2}\rho \|^2_{L^2}+\|(1+\rho)^{\frac 1 2}u\|^2_{L^2}+\|g\|^2_{L^2(\mathcal{L}^{2})}+2\eta\int_{\mathbb{R}^{d}} u\nabla\rho dx\right)  \\ \notag
&~~~~+ 2\left(\mu\|\nabla u\|^2_{L^2}+(\mu+\mu')\|divu\|^2_{L^2}+
\eta\gamma\|\nabla\rho\|^2_{L^2}+\|\nabla_q g\|^2_{L^2(\mathcal{L}^{2})} \right)  \\ \notag
&\lesssim  \left(\epsilon^{\frac{1}{2}} + \epsilon + \delta + \eta + \eta C_{\delta} \right)\left(\|\nabla \rho\|^2_{H^{s-1}} + \|\nabla u\|^2_{H^{s}} + \|\nabla_q g\|^2_{H^{s}(\mathcal{L}^{2})}\right).
\end{align}
\end{lemm}
\begin{proof}
	multiplying~ $\psi_\infty$ to system $\eqref{eq1}_3$ and integrating over $\mathbb{R}^d$ with $q$, we deduce that
	\begin{align}\label{1ineq1}
		\partial_t\int_{\mathbb{R}^d}g\psi_{\infty}dq + u\cdot\nabla\int_{\mathbb{R}^d}g\psi_{\infty}dq = 0~,
	\end{align}
	which implies
	\begin{align}\label{1ineq2}
		\int_{\mathbb{R}^d}g\psi_{\infty}dq = \int_{\mathbb{R}^d}g_0\psi_{\infty}dq = 0~~~.
	\end{align}
	Denote~$L^2(\mathcal{L}^{2})$ inner product by $\langle f,g\rangle=\int_{\mathbb{R}^{d}}\int_{\mathbb{R}^{d}}fg\psi_\infty dqdx$ . We infer from \eqref{1ineq2} that
	\begin{align}\label{ineq3}
		\langle {\rm div}u,g\rangle = \int_{\mathbb{R}^d}{\rm div}u \int_{\mathbb{R}^d}g\psi_{\infty}dqdx = 0~.~~
	\end{align}
	Taking $L^2(\mathcal{L}^{2})$ inner product with $g$ to system $\eqref{eq1}_3$ , we obtain
	\begin{align}\label{1ineq4}
		\frac {1} {2}\frac {d} {dt} \|g\|^2_{L^2(\mathcal{L}^{2})}+\|\nabla_q g\|^2_{L^2(\mathcal{L}^{2})}-\int_{\mathbb{R}^{d}}\nabla u:\tau dx  = -\langle u\cdot\nabla g, g\rangle-\langle\frac 1 {\psi_\infty} \nabla_q\cdot(\nabla uqg\psi_\infty), g \rangle~.
	\end{align}
	Integrating by parts and applying Lemma \ref{Lemma5} , we have
	\begin{align}\label{1ineq5}
		-\langle u\cdot\nabla g,g\rangle = \frac 1 2 \langle {\rm div}u, g^2\rangle
		&\lesssim \|\nabla u\|_{L^\infty}\|g\|^2_{L^2(\mathcal{L}^{2})} \\ \notag
		&\lesssim \|\nabla u\|_{L^\infty}\|\nabla_q g\|^2_{L^2(\mathcal{L}^{2})}~,~~~~~~~~~~~~
	\end{align}
 Integrating by part leads to
	\begin{align}\label{1ineq6}
		\langle\frac 1 {\psi_\infty} \nabla_q\cdot(\nabla uqg\psi_\infty), g \rangle&=
		-\int_{\mathbb{R}^{d}}\int_{\mathbb{R}^{d}}(\nabla uqg\psi_\infty)\nabla_q g dqdx\\ \notag
		&\lesssim \|\nabla u\|_{L^\infty}\|\nabla_q g\|^2_{L^2(\mathcal{L}^{2})}~.
	\end{align}
	 We deduce from \eqref{1ineq5} and \eqref{1ineq6} that
	\begin{align}\label{1ineq7}
		~~~~\frac {1} {2}\frac {d} {dt} \|g\|^2_{L^2(\mathcal{L}^{2})}+\|\nabla_q g\|^2_{L^2(\mathcal{L}^{2})}-\int_{\mathbb{R}^{d}}\nabla u:\tau dx\lesssim \|\nabla u\|_{L^\infty}\|\nabla_q g\|^2_{L^2(\mathcal{L}^{2})}~.
	\end{align}
	Let $h(\rho)=\frac {P'(1+\rho)} {1+\rho}$ and $i(\rho)=\frac 1 {\rho+1}$. Taking $L^2$ inner product wiht $h(\rho) \rho$ to system $\eqref{eq1}_1$ , we get
	\begin{align}\label{1ineq8}
		\frac 1 2 \frac {d} {dt} \int_{\mathbb{R}^{d}}h(\rho)|\rho|^2 dx&+\int_{\mathbb{R}^{d}}P'(1+\rho)\rho \rm divu dx \\ \notag
		&=\frac 1 2 \int_{\mathbb{R}^{d}} \partial_t h(\rho) |\rho|^2 dx-\int_{\mathbb{R}^{d}}h(\rho)\rho u\cdot\nabla \rho dx.~~~~~~~
	\end{align}
Taking $L^2$ inner product wiht $(1+\rho)u$ to system $\eqref{eq1}_2$ , we have
	\begin{align}\label{1ineq9}
		\frac 1 2 &\frac {d} {dt} \int_{\mathbb{R}^{d}}(1+\rho)|u|^2 dx+\int_{\mathbb{R}^{d}}P'(1+\rho)u\nabla\rho dx-\int_{\mathbb{R}^{d}}u \rm div\Sigma(u)dx  \\ \notag
		&=\frac 1 2 \int_{\mathbb{R}^{d}} \partial_t\rho|u|^2 dx-\int_{\mathbb{R}^{d}}u\cdot\nabla u (1+\rho)u dx+\int_{\mathbb{R}^{d}}u \rm div\tau dx.~~~~~~~~~~~
	\end{align}
	Applying Lemma \ref{Lemma3} leads to
	\begin{align}\label{1ineq10}
		\frac 1 2 &\int_{\mathbb{R}^{d}} \partial_t h(\rho) |\rho|^2 dx,\int_{\mathbb{R}^{d}}h(\rho)\rho u\cdot\nabla \rho dx\\ \notag
		&\lesssim \|\nabla\rho\|_{L^2}(\|\nabla \rho\|_{L^2}\|u\|_{L^d}+\|\nabla u\|_{L^2}\|\rho\|_{L^d})\\ \notag
		&\lesssim \|(\rho,u)\|_{L^d}(\|\nabla\rho\|^2_{L^2}+\|\nabla u\|^2_{L^2})~,~~~~~~~~~~~~~~~~~~~~~~~~
	\end{align}
	and
	\begin{align}\label{1ineq11}
		\frac 1 2& \int_{\mathbb{R}^{d}} \partial_t\rho|u|^2 dx + \int_{\mathbb{R}^{d}}u\cdot\nabla u (1+\rho)u dx \\ \notag
		&\lesssim \|\nabla \rho\|_{L^2}\|\nabla u\|_{L^2}\|u\|_{L^d}+\|\nabla u\|^2_{L^2}\|u\|_{L^d}(1+\|\rho\|_{L^\infty})\\ \notag
		&\lesssim (1+\|\rho\|_{L^\infty})\|(\rho,u)\|_{L^d}(\|\nabla\rho\|^2_{L^2}+\|\nabla u\|^2_{L^2})~.
	\end{align}
Similarly, we obtain
	\begin{align}\label{1ineq12}
		-\int_{\mathbb{R}^{d}}P'(1+\rho)(u\nabla\rho+\rho {\rm div}u)dx
		&=\int_{\mathbb{R}^{d}}P''(1+\rho)\rho u\nabla\rho dx \\ \notag
		&\lesssim \|\nabla \rho\|_{L^2}\|\nabla u\|_{L^2}\|\rho\|_{L^d}+\|\nabla \rho\|^2_{L^2}\|u\|_{L^d}\\ \notag
		&\lesssim \|(\rho,u)\|_{L^d}(\|\nabla\rho\|^2_{L^2}+\|\nabla u\|^2_{L^2})~.
	\end{align}
	Combining the estimates \eqref{1ineq10}-\eqref{1ineq12} and
	\begin{align}\label{1ineq13}
		\int_{\mathbb{R}^{d}}\nabla u:\tau dx = -\int_{\mathbb{R}^{d}}u\rm div\tau dx~,~~
	\end{align}
	we infer that
	\begin{align}\label{1ineq14}
		~~~~~\frac {d} {dt}& \left(\|h(\rho)^{\frac 1 2}\rho \|^2_{L^2}+\|(1+\rho)^{\frac 1 2}u\|^2_{L^2}+\|g\|^2_{L^2(\mathcal{L}^{2})}\right)  \\ \notag
		&~~~~+2\left(\mu\|\nabla u\|^2_{L^2}+(\mu+\mu')\|{\rm div}u\|^2_{L^2}+\|\nabla_q g\|^2_{L^2(\mathcal{L}^{2})}\right)  \\ \notag
		&\lesssim \left(1+\|\rho\|_{L^\infty}\right)\|(\rho,u)\|_{L^d}\left(\|\nabla\rho\|^2_{L^2}+\|\nabla u\|^2_{L^2}\right) + \|\nabla u\|_{L^\infty}\|\nabla_q g\|^2_{L^2(\mathcal{L}^{2})}~.
	\end{align}
Taking $L^2$ inner product wiht $\nabla\rho$ to system $\eqref{eq1}_2$ , we have
	\begin{align}\label{1ineq15}
		&\frac {d} {dt} \int_{\mathbb{R}^{d}} u\nabla\rho dx+\gamma\|\nabla\rho\|^2_{L^2}= \int_{\mathbb{R}^{d}}u\nabla\rho_t dx  \\ \notag
		&~~~~+\int_{\mathbb{R}^{d}}\nabla\rho\cdot\left(i(\rho) {\rm div}\Sigma{(u)}-(h(\rho)-\gamma) \nabla\rho-u\cdot\nabla u+i(\rho) {\rm div}\tau\right) dx ~~~ \\ \notag
		&=I_1+I_2~.
	\end{align}
	Integrating by part, we obtain
	\begin{align}\label{1ineq16}
		I_1=-\int_{\mathbb{R}^{d}}{\rm div}u\rho_t dx
		&\lesssim \|\nabla u\|^2_{L^2}(1+\|\rho\|_{L^\infty})+\|\nabla u\|_{L^2}\|\nabla \rho\|_{L^2}\|u\|_{L^\infty}~~~~~~~~\\ \notag
		&\lesssim  \|u\|_{L^\infty}\|\nabla\rho\|^2_{L^2} + (1+\|(\rho,u)\|_{L^\infty})\|\nabla u\|^2_{L^2}~.
	\end{align}
	Applying Lemma \ref{Lemma5} leads to
	\begin{align}\label{1ineq17}
		~~I_2&\lesssim \|\nabla \rho\|_{L^2}\left(\|\nabla^2 u\|_{L^2}+\|\rho\|_{L^\infty}\|\nabla \rho\|_{L^2}+\|u\|_{L^\infty}\|\nabla u\|_{L^2}+\|\nabla g\|_{L^2(\mathcal{L}^{2})}\right) \\ \notag
		&\lesssim \left(\delta+\|(\rho,u)\|_{L^\infty}\right)\left(\|\nabla\rho\|^2_{L^2}+\|\nabla u\|^2_{L^2}\right) + C_{\delta}\left(\|\nabla^2 u\|^2_{L^2}+\|\nabla_q\nabla g\|^2_{L^2(\mathcal{L}^{2})}\right)~.~~~~
	\end{align}
	where positive constant $\delta$ is small enough. According \eqref{1ineq16} and \eqref{1ineq17} , we infer that
	\begin{align}\label{1ineq18}
		\frac {d} {dt} \int_{\mathbb{R}^{d}} u\nabla\rho dx+\gamma\|\nabla\rho\|^2_{L^2} &\lesssim  C_{\delta}\left(\|\nabla^2 u\|^2_{L^2}+\|\nabla g\|^2_{L^2(\mathcal{L}^{2})}\right)\\ \notag
		&~~~~+ \left(\delta + \|(\rho,u)\|_{L^\infty}\right)\|\nabla\rho\|^2_{L^2}  + \left(1+\|(\rho,u)\|_{L^\infty}\right)\|\nabla u\|^2_{L^2}.
	\end{align}
According to \eqref{1ineq14} and \eqref{1ineq18} , we obtain
	\begin{align}\label{1ineq19}
		\frac {d} {dt} &\left(\|h(\rho)^{\frac 1 2}\rho \|^2_{L^2}+\|(1+\rho)^{\frac 1 2}u\|^2_{L^2}+\|g\|^2_{L^2(\mathcal{L}^{2})}+2\eta\int_{\mathbb{R}^{d}} u\nabla\rho dx\right)  \\ \notag
		&~~~~+2\left(\mu\|\nabla u\|^2_{L^2}+(\mu+\mu')\|{\rm div}u\|^2_{L^2}+
		\eta\gamma\|\nabla\rho\|^2_{L^2}+\|\nabla_q g\|^2_{L^2(\mathcal{L}^{2})}\right)  \\ \notag
		&\lesssim (1+\|\rho\|_{L^\infty})\|(\rho,u)\|_{L^d}(\|\nabla\rho\|^2_{L^2}+\|\nabla u\|^2_{L^2}) + \|\nabla u\|_{L^\infty}\|\nabla_q g\|^2_{L^2(\mathcal{L}^{2})} \\ \notag
		& ~~~~ + \eta(\delta+\|(\rho,u)\|_{L^\infty})\|\nabla\rho\|^2_{L^2}+\eta(1+\|(\rho,u)\|_{L^\infty})\|\nabla u\|^2_{L^2}  \\ \notag
		&~~~~+ \eta C_{\delta}\left(\|\nabla^2 u\|_{L^2}+\|\nabla g\|_{L^2(\mathcal{L}^{2})}\right) \\ \notag
		&\lesssim \left(\epsilon^{\frac{1}{2}} + \epsilon + \delta + \eta + \eta C_{\delta} \right)\left(\|\nabla \rho\|^2_{H^{s-1}} + \|\nabla u\|^2_{H^{s}} + \|\nabla_q g\|^2_{H^{s}(\mathcal{L}^{2})}\right)~.~~~~~~~~~~
	\end{align}
	where positive constant $\eta$ is small enough.	We thus complete the proof of Lemma \ref{low} .
\end{proof}
\subsection{Estimates on high order derivatives}
\begin{lemm}\label{high}
Let $(\rho_,u,g)$ be local classical solutions considered in Proposition \ref{pro2} . Then there exist positive constants $\eta$ and $\delta$ such that
\begin{align}\label{high estimate}
	\frac {d} {dt}& \left(\|h(\rho)^{\frac 1 2}\Lambda^s\rho \|^2_{L^2}+\|(1+\rho)^{\frac 1 2}\Lambda^su\|^2_{L^2}+\|\Lambda^sg\|^2_{L^2(\mathcal{L}^{2})}+2\eta\int_{\mathbb{R}^{d}} \Lambda^{s-1} u\nabla\Lambda^{s-1} \rho dx\right)  \\ \notag
	&~~~~+2\left(\mu\|\nabla\Lambda^{s} u\|^2_{L^2}+(\mu+\mu')\|{\rm div}\Lambda^su\|^2_{L^2}+
	\eta\gamma\|\nabla\Lambda^{s-1}\rho\|^2_{L^2}+\|\nabla_q \Lambda^sg\|^2_{L^2(\mathcal{L}^{2})}\right)  \\ \notag
	&\lesssim  \left(\epsilon^{\frac{1}{2}}+ \left(\frac{\epsilon}{\lambda}\right)^{\frac{1}{2}}  + \epsilon + \delta + \eta + \eta C_{\delta} \right)\left(\|\nabla \rho\|^2_{H^{s-1}} + \|\nabla u\|^2_{H^{s}} + \|\nabla_q g\|^2_{H^{s}(\mathcal{L}^{2})}\right).
\end{align}
\end{lemm}
\begin{proof}
Applying $\Lambda^s$ to system $\eqref{eq1}_3$ , we infer that
\begin{align}\label{2ineq1}
	\partial_t\Lambda^s g+\mathcal{L}\Lambda^s g&+{\rm div}\Lambda^{s}u+\nabla\Lambda^{s}uq\nabla_q \mathcal{U} \\ \notag
	&=-u\cdot\nabla\Lambda^s g-[\Lambda^s,u]\nabla g \\ \notag
	&~~~~-\frac 1 {\psi_\infty}\nabla_q \cdot\left(\Lambda^s\nabla uqg\psi_\infty+q\psi_\infty[\Lambda^s,g]\nabla u\right)~.
\end{align}
Taking $L^2(\mathcal{L}^{2})$ inner product with $\Lambda^s g $ to \eqref{2ineq1} , we get
\begin{align}\label{2ineq2}
	\frac {1} {2}\frac {d} {dt} \|\Lambda^s g\|^2_{L^2(\mathcal{L}^{2})}&+\|\nabla_q  \Lambda^s g\|^2_{L^2(\mathcal{L}^{2})}-\int_{\mathbb{R}^{d}}\nabla\Lambda^{s}u:\Lambda^{s}\tau dx \\ \notag
	&= -\langle u\cdot\nabla \Lambda^s g,\Lambda^s g\rangle  -\langle[\Lambda^s,u]\nabla g,\Lambda^s g\rangle \\ \notag
	&~~~~-\langle\frac 1 {\psi_\infty} \nabla_q\cdot(\Lambda^s\nabla uqg\psi_\infty),\Lambda^sg\rangle\\ \notag
	&~~~~-\langle\frac 1 {\psi_\infty} \nabla_q\cdot(q\psi_\infty[\Lambda^s,g]\nabla u),\Lambda^s g\rangle,~~~~~~~~~~
\end{align}
where $\langle{\rm div}\Lambda^su,\Lambda^s g\rangle = 0$ has been used.
Applying Lemmas \ref{Lemma5} and \ref{Lemma6} lead to
\begin{align}\label{2ineq3}
	-\langle[\Lambda^s,u]\nabla g,\Lambda^s g\rangle \lesssim \|u\|_{H^s}\|\nabla_q\nabla g\|^2_{H^{s-1}(\mathcal{L}^{2})}~,~~~~~~~~~~~~~~~~~~~~~
\end{align}
and
\begin{align}\label{2ineq4}
	-\langle u\cdot\nabla \Lambda^s g,\Lambda^s g\rangle=\frac 1 2 \langle {\rm div}u, (\Lambda^s g)^2\rangle
	&\lesssim \|\nabla u\|_{L^\infty}\|\Lambda^s g\|^2_{L^2(\mathcal{L}^{2})} \\ \notag
	&\lesssim \|u\|_{H^s}\|\nabla_q\Lambda^s g\|^2_{L^2(\mathcal{L}^{2})}~.
\end{align}
Analogously,
\begin{align}\label{2ineq5}
	\langle\frac 1 {\psi_\infty} \nabla_q\cdot(\Lambda^s\nabla uqg\psi_\infty), \Lambda^sg \rangle
	&=-\int_{\mathbb{R}^{d}}\int_{\mathbb{R}^{d}}(\Lambda^s\nabla uq\psi_\infty g)\nabla_q \Lambda^s g dqdx \\ \notag
	&\lesssim \|\nabla \Lambda^s u\|_{L^2}\|\langle q\rangle g\|_{H^{s-1}(\mathcal{L}^{2})}\|\nabla_q\Lambda^s g\|_{L^2(\mathcal{L}^{2})} \\ \notag
	&\lesssim   \|\langle q\rangle g\|_{H^{s-1}(\mathcal{L}^{2})}\left(\|\nabla \Lambda^s u\|^2_{L^2}+\|\nabla_q\Lambda^s g\|^2_{L^2(\mathcal{L}^{2})}\right)~,
\end{align}
and
\begin{align}\label{2ineq6}
	-\langle\frac 1 {\psi_\infty} \nabla_q\cdot(q\psi_\infty[\Lambda^s,g]\nabla u),\Lambda^s g\rangle&=\langle [\Lambda^s,q g]\nabla u,\nabla_q\Lambda^s g\rangle\\ \notag
	&\lesssim \|\nabla_q \Lambda^s g\|_{L^2(\mathcal{L}^{2})}\|u\|_{H^s}\|\nabla\nabla_q g\|_{H^{s-1}(\mathcal{L}^{2})}~~~~~~~~~~~~~~~~~~\\ \notag
	&\lesssim \|u\|_{H^s}\|\nabla_q g\|^2_{H^s(\mathcal{L}^{2})}~.
\end{align}
Combining the estimates \eqref{2ineq3}-\eqref{2ineq6} , we obtain
\begin{align}\label{2ineq7}
	\frac {1} {2}\frac {d} {dt} \|\Lambda^s g\|^2_{L^2(\mathcal{L}^{2})}&+\|\nabla_q \Lambda^s g\|^2_{L^2(\mathcal{L}^{2})} -\int_{\mathbb{R}^{d}}\nabla\Lambda^{s}u:\Lambda^{s}\tau dx \\ \notag
	&\lesssim \left(\|u\|_{H^s}+\|\langle q\rangle g\|_{H^{s-1}(\mathcal{L}^{2})}\right)\left(\|\nabla \Lambda^s u\|^2_{L^2}+\|\nabla_q g\|^2_{H^s(\mathcal{L}^{2})}\right) \\ \notag
	&\lesssim \epsilon^{\frac{1}{2}}\left( 1+\lambda^{-\frac{1}{2}} \right)\left(\|\nabla \Lambda^s u\|^2_{L^2}+\|\nabla_q g\|^2_{H^s(\mathcal{L}^{2})}\right).
\end{align}
Applying $\Lambda^s$ to system $\eqref{eq1}_1$ , one can get
\begin{align}\label{2ineq8}
	\partial_t\Lambda^s \rho+{\rm div} \Lambda^s u(1+\rho)
	=-u\cdot\nabla\Lambda^s \rho-[\Lambda^s,u]\nabla\rho-[\Lambda^s,\rho]{\rm div}u.~~~~~~~~~~~~~~~~
\end{align}
Taking $L^2$ inner product with $h(\rho) \Lambda^s \rho$ to \eqref{2ineq8} , we have
\begin{align}\label{2ineq9}
	\frac 1 2 &\frac {d} {dt} \int_{\mathbb{R}^{d}}h(\rho)|\Lambda^s \rho|^2 dx+\int_{\mathbb{R}^{d}}P'(1+\rho)\Lambda^s \rho {\rm div}\Lambda^s u dx \\ \notag &=\frac 1 2 \int_{\mathbb{R}^{d}} \partial_t h(\rho) |\Lambda^s \rho|^2 dx -\int_{\mathbb{R}^{d}}\Lambda^s \rho \cdot h(\rho)  u\cdot\nabla \Lambda^s \rho dx \\ \notag
	& ~~~~-\int_{\mathbb{R}^{d}}[\Lambda^s,u]\nabla\rho\cdot h(\rho) \Lambda^s\rho dx-\int_{\mathbb{R}^{d}}[\Lambda^s,(1+\rho)]{\rm div}u\cdot h(\rho) \Lambda^s \rho dx~.~~~~~~~~
\end{align}
Applying $\Lambda^m$ to system $\eqref{eq1}_2$ , we infer that
\begin{align}\label{2ineq10}
	\partial_t\Lambda^m u+h(\rho)\nabla\Lambda^m\rho&-i(\rho) {\rm div}\Lambda^m \Sigma{(u)}-i(\rho) {\rm div}\Lambda^m \tau  \\ \notag
	&=-u\cdot\nabla\Lambda^m u-[\Lambda^m,u]\nabla u-[\Lambda^m,h(\rho) -\gamma]\nabla\rho\\ \notag &~~~~+[\Lambda^m,i(\rho)-1]{\rm div}\Sigma{(u)}+[\Lambda^m,i(\rho)-1]{\rm div}\tau.~~~~~~~~~
\end{align}
Taking $L^2$ inner product with $(1+\rho) \Lambda^s u$ to \eqref{2ineq10} with $m=s$ , we obtain
\begin{align}\label{2ineq11}
	~~~\frac 1 2 	&\frac {d} {dt} \|(1+\rho)^{\frac 1 2}\Lambda^s u\|^2_{L^2}
	+\mu\|\nabla\Lambda^s u\|^2_{L^2}+(\mu+\mu')\|{\rm div}\Lambda^s u\|^2_{L^2}  \\ \notag
	&=-\int_{\mathbb{R}^{d}}P'(1+\rho)\nabla\Lambda^s \rho \Lambda^s u dx+\int_{\mathbb{R}^{d}}{\rm div}\Lambda^s \tau \Lambda^s u dx \\ \notag
	&~~~~+\frac 1 2 \int_{\mathbb{R}^{d}} \partial_t\rho |\Lambda^s u|^2 dx
	-\int_{\mathbb{R}^{d}}\Lambda^s u \cdot(1+\rho)u\cdot\nabla \Lambda^s u dx \\ \notag
	&~~~~-\int_{\mathbb{R}^{d}}[\Lambda^s,u]\nabla u (1+\rho)\Lambda^s u dx
	-\int_{\mathbb{R}^{d}}[\Lambda^s,h(\rho)-\gamma]\nabla\rho (1+\rho)\Lambda^s u dx \\ \notag
	&~~~~+\int_{\mathbb{R}^{d}}[\Lambda^s,i(\rho)-1]{\rm div}\Sigma{(u)} (1+\rho)\Lambda^s u dx
	+\int_{\mathbb{R}^{d}}[\Lambda^s,i(\rho)-1]{\rm div}\tau (1+\rho)\Lambda^s u dx~.
\end{align}
Integrating by parts leads to
\begin{align}\label{2ineq12}
	\frac 1 2
	\int_{\mathbb{R}^{d}} \partial_t h(\rho) |\Lambda^s \rho|^2 dx
	\lesssim \|(\rho,u)\|_{H^s}\|\Lambda^s \rho\|^2_{L^2}~,
\end{align}
and
\begin{align}\label{2ineq13}
	-\int_{\mathbb{R}^{d}}\Lambda^s \rho \cdot h(\rho)  u\cdot\nabla \Lambda^s \rho dx
	&=\frac 1 2 \int_{\mathbb{R}^{d}} {\rm div}(h(\rho)u)|\Lambda^s \rho|^2 dx\\ \notag
	&\lesssim \|(\rho,u)\|_{H^s}\|\Lambda^s \rho\|^2_{L^2}~.
\end{align}
It follows from Lemma \ref{Lemma6} that
\begin{align}\label{2ineq14}
	-\int_{\mathbb{R}^{d}}[\Lambda^s,u]\nabla\rho\cdot h(\rho) \Lambda^s\rho dx&-\int_{\mathbb{R}^{d}}[\Lambda^s,\rho]{\rm div}u\cdot h(\rho) \Lambda^s \rho dx  \\ \notag
	&\lesssim (\|\rho\|_{H^s}\|\Lambda^{s-1}\nabla u\|_{L^2}+\|\Lambda^{s-1}\nabla \rho\|_{L^2}\|u\|_{H^s})\|\Lambda^s \rho\|_{L^2}~~~ \\ \notag
	&\lesssim \left(\|\rho\|_{H^s}+\|u\|_{H^s}\right)\left(\|\Lambda^s \rho\|^2_{L^2}+\|\Lambda^{s} u\|^2_{L^2}\right)~,
\end{align}
and
\begin{align}\label{2ineq15}
	\frac 1 2 \int_{\mathbb{R}^{d}} \partial_t\rho |\Lambda^s u|^2 dx
	&-\int_{\mathbb{R}^{d}}\Lambda^s u \cdot(1+\rho)u\cdot\nabla \Lambda^s u dx
	-\int_{\mathbb{R}^{d}}[\Lambda^s,u]\nabla u (1+\rho)\Lambda^s u dx  \\ \notag
	&\lesssim \|u\|_{H^s}\|\Lambda^s u\|_{L^2}\left(\|\Lambda^s u\|_{L^2}+\|\nabla\Lambda^s u\|_{L^2}+\|\nabla\Lambda^{s-1} u\|_{L^2}\right) \\ \notag
	&\lesssim \|u\|_{H^s}\left(\|\Lambda^s u\|^2_{L^2}+\|\nabla\Lambda^s u\|^2_{L^2}\right)~.
\end{align}
Similarly, we have
\begin{align}\label{2ineq16}
	-\int_{\mathbb{R}^{d}}[\Lambda^s,h(\rho)-\gamma]\nabla\rho (1+\rho)\Lambda^s u dx
	&+\int_{\mathbb{R}^{d}}[\Lambda^s,i(\rho)-1]{\rm div}\Sigma{(u)} (1+\rho)\Lambda^s u dx  \\ \notag
	&\lesssim\|\rho\|_{H^s}\|\Lambda^s u\|_{L^2}\left(\|\nabla\Lambda^{s-1} \rho\|_{L^2}+\|\nabla^2 u\|_{H^{s-1}}\right) \\ \notag
	&\lesssim\|\rho\|_{H^s}\left(\|\nabla\Lambda^{s-1} \rho\|^2_{L^2}+\|\nabla u\|^2_{H^{s}}\right)~,
\end{align}
and
\begin{align}\label{2ineq17}
	\int_{\mathbb{R}^{d}}[\Lambda^s,i(\rho)-1]{\rm div}\tau (1+\rho)\Lambda^s u dx&\lesssim \|\rho\|_{H^s}\|\Lambda^s u\|_{L^2}\|\nabla \nabla_q g\|_{H^{s-1}(\mathcal{L}^{2})} \\ \notag
	& \lesssim \|\rho\|_{H^s}\left(\|\Lambda^s u\|^2_{L^2} + \|\nabla_q g\|^2_{H^{s}(\mathcal{L}^{2})}\right)~.~~
\end{align}
Integrating by parts leads to
\begin{align}\label{2ineq18}
	-\int_{\mathbb{R}^{d}}P'(1+\rho)(\Lambda^s u\nabla\Lambda^s\rho+\Lambda^s\rho {\rm div}\Lambda^su)dx
	&=\int_{\mathbb{R}^{d}}P''(1+\rho)\Lambda^s\rho \Lambda^s u\nabla\rho dx \\ \notag
	&\lesssim \|\rho\|_{H^s}\|\Lambda^s u\|_{L^2}\|\Lambda^s \rho\|_{L^2} \\ \notag
	&\lesssim \|\rho\|_{H^s}\left(\|\Lambda^s \rho\|^2_{L^2} + \|\Lambda^s u\|^2_{L^2}\right)~.~~~
\end{align}
Combining the estimates \eqref{2ineq12}-\eqref{2ineq18} , we deduce that
\begin{align}\label{2ineq19}
	\frac {d} {dt}& \left(\|h(\rho)^{\frac 1 2}\Lambda^s\rho \|^2_{L^2}+\|(1+\rho)^{\frac 1 2}\Lambda^s u\|^2_{L^2}+\|\Lambda^s g\|^2_{L^2(\mathcal{L}^{2})}\right)  \\ \notag
	&~~~~+2\left(\mu\|\nabla \Lambda^s u\|^2_{L^2}+(\mu+\mu')\|{\rm div}\Lambda^s u\|^2_{L^2}+\|\nabla_q \Lambda^s g\|^2_{L^2(\mathcal{L}^{2})}\right)  ~~~~~~~~~~~~\\ \notag
	&\lesssim \|(\rho,u)\|_{H^s}\left(\|\Lambda^s \rho\|^2_{L^2} + \|\nabla u\|^2_{H^{s}} + \|\nabla_q g\|^2_{H^{s}(\mathcal{L}^{2})}\right)~.
\end{align}
Taking $L^2$ inner product with $\nabla\Lambda^{s-1} \rho$ to \eqref{2ineq10} with $m=s-1$ , we get
\begin{align}\label{2ineq20}
	\frac {d} {dt} &\int_{\mathbb{R}^{d}}\Lambda^{s-1} u \cdot\nabla\Lambda^{s-1} \rho dx
	+\gamma\|\nabla\Lambda^{s-1}\rho\|^2_{L^{2}}\\ \notag
	&=-\int_{\mathbb{R}^{d}} \Lambda^{s-1} \rho_t {\rm div}\Lambda^{s-1} u dx
	-\int_{\mathbb{R}^{d}}\nabla\Lambda^{s-1} \rho\cdot u\cdot\nabla \Lambda^{s-1} u dx\\ \notag
	&~~~~-\int_{\mathbb{R}^{d}}[\Lambda^{s-1},u]\nabla u \nabla\Lambda^{s-1} \rho dx
	-\int_{\mathbb{R}^{d}}\Lambda^{s-1}((h(\rho)-\gamma)\nabla\rho) \nabla\Lambda^{s-1} \rho dx\\ \notag
	&~~~~+\int_{\mathbb{R}^{d}}\Lambda^{s-1}(i(\rho){\rm div}\Sigma{(u)}) \nabla\Lambda^{s-1} \rho dx
	+\int_{\mathbb{R}^{d}}\Lambda^{s-1}(i(\rho){\rm div}\tau) \nabla\Lambda^{s-1} \rho dx~.
\end{align}
Lemmas \ref{Lemma5} enable us to obtain
\begin{align}\label{2ineq21}
	-\int_{\mathbb{R}^{d}} \Lambda^{s-1} \rho_t {\rm div}\Lambda^{s-1} u dx	&-\int_{\mathbb{R}^{d}}\nabla\Lambda^{s-1} \rho\cdot u\cdot\nabla \Lambda^{s-1} u dx \\ \notag
	&\lesssim \|u\|_{H^s}\left(\|\nabla\rho\|^2_{H^{s-1}}+\|\nabla u\|^2_{H^{s-1}}\right)~,~~~~~~~~~~~~~
\end{align}
and
\begin{align}\label{2ineq22}
	\int_{\mathbb{R}^{d}}\Lambda^{s-1}(i(\rho){\rm div}\tau) \nabla\Lambda^{s-1} \rho dx&\lesssim \|\nabla\Lambda^{s-1} \rho\|_{L^2}\|\nabla \nabla_q g\|_{H^{s-1}(\mathcal{L}^{2})}\left(\|\rho\|_{H^{s-1}}+1\right)\\ \notag
	&\lesssim \left(\|\rho\|_{H^{s-1}}+1\right)\left(\delta\|\nabla\Lambda^{s-1} \rho\|^2_{L^2} + C_{\delta}\|\nabla \nabla_q g\|^2_{H^{s-1}(\mathcal{L}^{2})}\right)~,
\end{align}
with sufficiently small positive constant $\delta$. Similarly, we have
\begin{align}\label{2ineq23}
	-\int_{\mathbb{R}^{d}}\Lambda^{s-1}&[(h(\rho)-\gamma)\nabla\rho] \nabla\Lambda^{s-1} \rho dx +\int_{\mathbb{R}^{d}}\Lambda^{s-1}[i(\rho){\rm div}\Sigma{(u)}] \nabla\Lambda^{s-1} \rho dx\\ \notag
	&\lesssim \|\nabla\Lambda^{s-1} \rho\|_{L^2}(\|\rho\|_{H^{s}}\|\nabla\Lambda^{s-1} \rho\|_{L^2}+\|\nabla^2 u\|_{H^{s-1}}+\|\nabla^2 u\|_{H^{s-1}}\|\rho\|_{H^{s-1}}) ~~~~~~\\ \notag
	&\lesssim \left(\delta + \|\rho\|_{H^s}\right)\|\nabla\Lambda^{s-1} \rho\|^2_{L^2} + \left(C_{\delta} + \|\rho\|_{H^s}\right)\|\nabla^2 u\|^2_{H^{s-1}}~.
\end{align}
Applying Lemmas \ref{Lemma3} and \ref{Lemma6} lead to
\begin{align}\label{2ineq24}
	-\int_{\mathbb{R}^{d}}[\Lambda^{s-1},u]\nabla u \nabla\Lambda^{s-1} \rho dx&\lesssim \|\nabla\Lambda^{s-1} \rho\|_{L^2}\|\nabla u\|_{L^{\infty}}\|\nabla\Lambda^{s-2} u\|_{L^2}\\ \notag
	&\lesssim \|\nabla\Lambda^{s-1} \rho\|_{L^2}\|u\|^{1-\frac 1 s-\frac d {2s}}_{L^{2}}\|\nabla\Lambda^{s-1} u\|^{\frac 1 s+\frac d {2s}}_{L^2}\|u\|^{\frac 1 s}_{L^{2}}\|\nabla\Lambda^{s-1} u\|^{1-\frac 1 s}_{L^2}\\ \notag
	&\lesssim \|u\|_{H^{s}}\left(\|\nabla\Lambda^{s-1} \rho\|^2_{L^2} + \|\nabla\Lambda^{s-1} u\|^2_{L^2}\right)~.
\end{align}
Combining the estimates \eqref{2ineq21}-\eqref{2ineq24} , we deduce that
\begin{align}\label{2ineq25}
	\frac {d} {dt} \int_{\mathbb{R}^{d}}\Lambda^{s-1} u \cdot\nabla\Lambda^{s-1} \rho& dx
	+\gamma\|\nabla\Lambda^{s-1}\rho\|^2_{L^{2}} \\ \notag
	&\lesssim  \|(\rho,u)\|_{H^s}\left(\|\nabla\rho\|^2_{H^{s-1}}  + \|\nabla u\|^2_{H^{s}} + \|\nabla \nabla_q g\|^2_{H^{s-1}(\mathcal{L}^{2})}\right)~ \\ \notag
	&~~~~+ \delta\|\nabla\rho\|^2_{H^{s-1}} +  C_{\delta}\left(\|\nabla u\|^2_{H^{s}} + \|\nabla \nabla_q g\|^2_{H^{s-1}(\mathcal{L}^{2})}\right)~.
\end{align}
 According to \eqref{2ineq19} and \eqref{2ineq25} , we infer that
\begin{align}\label{2ineq26}
	\frac {d} {dt} &(\|h(\rho)^{\frac 1 2}\Lambda^s\rho \|^2_{L^2}+\|(1+\rho)^{\frac 1 2}\Lambda^s u\|^2_{L^2}+\|\Lambda^s g\|^2_{L^2(\mathcal{L}^{2})}+2\eta\int_{\mathbb{R}^{d}} \Lambda^{s-1} u\nabla\Lambda^{s-1} \rho dx)  \\ \notag
	&~~~+2(\mu\|\nabla \Lambda^s u\|^2_{L^2}+(\mu+\mu')\|{\rm div}\Lambda^s u\|^2_{L^2}+
	\eta\gamma\|\nabla\Lambda^{s-1}\rho\|^2_{L^2}+\|\nabla_q \Lambda^s g\|^2_{L^2(\mathcal{L}^{2})})  \\ \notag
	&\lesssim \|(\rho,u)\|_{H^s}\left(\|\Lambda^s \rho\|^2_{L^2} + \|\nabla u\|^2_{H^{s}} + \|\nabla_q g\|^2_{H^{s}(\mathcal{L}^{2})}\right) \\ \notag
	& ~~~~+ \eta\|(\rho,u)\|_{H^s}\left(\|\nabla\rho\|^2_{H^{s-1}}  + \|\nabla u\|^2_{H^{s}} + \|\nabla \nabla_q g\|^2_{H^{s-1}(\mathcal{L}^{2})}\right) \\ \notag
	&~~~~+ \eta\delta\|\nabla\rho\|^2_{H^{s-1}} +  \eta C_{\delta}\left(\|\nabla u\|^2_{H^{s}} + \|\nabla \nabla_q g\|^2_{H^{s-1}(\mathcal{L}^{2})}\right) \\ \notag
	&\lesssim  \left(\epsilon^{\frac{1}{2}}+\left(\frac{\epsilon}{\lambda}\right)^{\frac{1}{2}} + \epsilon + \delta + \eta + \eta C_{\delta} \right)\left(\|\nabla \rho\|^2_{H^{s-1}} + \|\nabla u\|^2_{H^{s}} + \|\nabla_q g\|^2_{H^{s}(\mathcal{L}^{2})}\right),
\end{align}
where positive constant $\eta$ is sufficiently small. We thus complete the proof of Lemma \ref{high} .
\end{proof}
\subsection{Estimates with micro weight}
\par
Once we estimate $\|q g\|_{H^{s-1}(\mathcal{L}^{2})}$ as $\|\nabla_q g\|_{H^{s-1}(\mathcal{L}^{2})}$ instead of $\|\langle q \rangle g\|_{H^{s-1}(\mathcal{L}^{2})}$, the terms behaving in $D^{\frac{3}{2}}$ would appear on the right hand side and result in the failure to obtain \eqref{decayenergy}. In order to obtain the estimate in the shape of $\frac{d}{dt}E + D \leq E^{\theta}D$ for some $\theta>0$, the term $\|\nabla \Lambda^s u\|_{L^2}\|\langle q \rangle g\|_{H^{s-1}(\mathcal{L}^{2})}\|\nabla_q\Lambda^s g\|_{L^2(\mathcal{L}^{2})}$ forces us to consider the term $\|\langle q \rangle g\|_{H^{s-1}(\mathcal{L}^{2})}$ acting as $E$ instead.
\begin{lemm}\label{mix}
	Let $(\rho_,u,g)$ be local classical solutions considered in Proposition \ref{pro2} . Then there exist a positive constant $\delta$ such that
	\begin{align}\label{mix estimate}
		\frac {d} {dt} \|\langle q \rangle g\|^2_{H^{s-1}(\mathcal{L}^{2})}  &+ 2\|\langle q \rangle \nabla_q g\|^2_{H^{s-1}(\mathcal{L}^{2})}  \\ \notag 
		&\lesssim C\|\nabla_q g\|^2_{H^{s-1}(\mathcal{L}^{2})} + C_{\delta}\|\nabla u\|^2_{H^{s-1}}+(\epsilon^{\frac{1}{2}}+\delta)\|\langle q\rangle\nabla_qg\|^2_{H^{s-1}(\mathcal{L}^{2})}.
	\end{align}
\end{lemm}
\begin{proof}
	Applying $\Lambda^{s-1}$ to system $\eqref{eq1}_3$ and taking $L^2(\mathcal{L}^2)$ inner product with $\langle q \rangle^2 \Lambda^{s-1} g$, we infer that
	\begin{align}\label{3ineq1}
		\frac {1} {2}\frac {d} {dt} \|\langle q \rangle \Lambda^{s-1} g\|^2_{L^2(\mathcal{L}^{2})} &- \langle \mathcal{L}\Lambda^{s-1} g,\langle q \rangle^2 \Lambda^{s-1} g\rangle \\ \notag
		&= -\langle u\cdot\nabla \Lambda^{s-1} g,\langle q \rangle^2 \Lambda^{s-1} g\rangle-\langle[\Lambda^{s-1},u\cdot\nabla] g,\langle q \rangle^2 \Lambda^{s-1} g\rangle~~~~\\ \notag
		&~~~~-\langle{\rm div}\Lambda^{s-1}u,\langle q\rangle^2\Lambda^{s-1}g\rangle+\langle \nabla \Lambda^{s-1} u q\nabla_{q}\mathcal{U},\langle q \rangle^2 \Lambda^{s-1} g\rangle \\ \notag
		&~~~~-\langle\frac 1 {\psi_\infty} \nabla_q\cdot(\Lambda^{s-1}(\nabla uqg)\psi_\infty),\langle q \rangle^2 \Lambda^{s-1}g\rangle~.
	\end{align}
	According to Lemma \ref{Lemma5} , we get
	\begin{align}\label{3ineq2}
		-\langle \mathcal{L}\Lambda^{s-1} g,\langle q \rangle^2 \Lambda^{s-1} g\rangle &= \|\langle q \rangle \nabla_q \Lambda^{s-1} g\|^2_{L^2(\mathcal{L}^{2})} + 2 \langle \nabla_q \Lambda^{s-1} g,q \Lambda^{s-1}g\rangle \\ \notag
		&\gtrsim \|\langle q \rangle \nabla_q \Lambda^{s-1} g\|^2_{L^2(\mathcal{L}^{2})} - C\|\nabla_q \Lambda^{s-1} g\|^2_{L^2(\mathcal{L}^{2})}.~~~~~~~~~~~~~~~~~~~~~
	\end{align}
	One can get from condition \eqref{potential1} and Lemma \ref{Lemma4} that
		\begin{align}\label{3ineq3'}
		~\langle {\rm div}\Lambda^{s-1}u,\langle q \rangle^2 \Lambda^{s-1} g\rangle 
		&\lesssim \int_{\mathbb{R}^d}|{\rm div} \Lambda^{s-1} u|\cdot\|\langle q \rangle^2 \Lambda^{s-1} g\|_{\mathcal{L}^2}dx\\ \notag
		&\lesssim C_{\delta}\|{\rm div} \Lambda^{s-1} u\|^2_{L^2}+\delta\|\langle q\rangle\nabla_q\Lambda^{s-1} g\|^2_{L^2(\mathcal{L}^{2})}~,
	\end{align}
and
	\begin{align}\label{3ineq3}
		\langle \nabla \Lambda^{s-1} u q\nabla_{q}\mathcal{U},\langle q \rangle^2 \Lambda^{s-1} g\rangle &= \int_{\mathbb{R}^d}\nabla \Lambda^{s-1} u\int_{\mathbb{R}^d}q\nabla_{q}\mathcal{U}\langle q \rangle^2 \Lambda^{s-1} g\psi_{\infty}dqdx\\ \notag
		&\lesssim \int_{\mathbb{R}^d}|\nabla \Lambda^{s-1} u|\cdot\|\langle q \rangle^2 \Lambda^{s-1} g\|_{\mathcal{L}^2}dx\\ \notag
		&\lesssim C_{\delta}\|\nabla \Lambda^{s-1} u\|^2_{L^2}+\delta\|\langle q\rangle\nabla_q\Lambda^{s-1} g\|^2_{L^2(\mathcal{L}^{2})}~,~~~~~~
	\end{align}
	where positive constant $\delta$ is sufficiently small. Applying Lemma \ref{Lemma4} , we have
	\begin{align}\label{3ineq4}
		\langle u\cdot\nabla \Lambda^{s-1} g,\langle q \rangle^2 \Lambda^{s-1} g\rangle &= \int_{\mathbb{R}^d} {\rm divu} \|\langle q\rangle\Lambda^{s-1}g\|^2_{\mathcal{L}^2}dx\\ \notag
		&\lesssim \|\nabla u\|_{H^{s-1}}\|\langle q\rangle\nabla_q\Lambda^{s-1}g\|^2_{L^2(\mathcal{L}^{2})}.~~~~~~~~~~~~~~~~
	\end{align}
	According to Lemmas \ref{4lemma4} and \ref{4lemma6} , one can arrive
	\begin{align}\label{3ineq5}
		\langle[\Lambda^{s-1},u\cdot\nabla] g,\langle q \rangle^2 \Lambda^{s-1} g\rangle &\lesssim
		\|u\|_{H^s}\|g\|_{H^s(\mathcal{L}^{2})}\|\langle q\rangle^2\Lambda^{s-1}g\|_{L^2(\mathcal{L}^{2})} \\ \notag
		&\lesssim \|u\|_{H^s}\left(\|g\|^2_{H^{s}}+\|\langle q\rangle\nabla_q\Lambda^{s-1}g\|^2_{L^2(\mathcal{L}^{2})}\right).~~~~~~~~~~~~~~~~
	\end{align}
	Integrating by parts leads to
	\begin{align}\label{3ineq6}
		-\langle\frac 1 {\psi_\infty} \nabla_q\cdot(\Lambda^{s-1}(\nabla uqg)\psi_\infty),\langle q \rangle^2 \Lambda^{s-1}g\rangle &= \langle\Lambda^{s-1}(\nabla uqg),\langle q \rangle^2 \nabla_q\Lambda^{s-1}g+2q\Lambda^{s-1}g\rangle ~~~\\ \notag
		& = J_1 +J_2~.
	\end{align}
 Lemmas \ref{Lemma4} ensure that
	\begin{align}\label{3ineq7}
		J_1 &= \langle\Lambda^{s-1}(\nabla uqg),\langle q \rangle^2 \nabla_q\Lambda^{s-1}g\rangle  \\ \notag
		&\lesssim \left(\|\nabla \Lambda^{s-1} u\|_{L^2}\|\langle q \rangle^2g\|_{L^{\infty}(\mathcal{L}^2)} + \|\nabla u\|_{L^{\infty}}\|\langle q \rangle^2\Lambda^{s-1}g\|_{L^{2}(\mathcal{L}^2)}\right)\|\langle q \rangle\nabla_q \Lambda^{s-1} g\|_{L^{2}(\mathcal{L}^2)}\\ \notag
		&\lesssim\|u\|_{H^s}\|\langle q \rangle \nabla_q g\|^2_{H^{s-1}(\mathcal{L}^2)}~,
	\end{align}
	and
	\begin{align}\label{3ineq8}
		J_2 &= \langle\Lambda^{s-1}(\nabla uqg),2q\Lambda^{s-1}g\rangle  \\ \notag
		&\lesssim \left(\|\nabla \Lambda^{s-1} u\|_{L^2}\|qg\|_{L^{\infty}(\mathcal{L}^2)} + \|\nabla u\|_{L^{\infty}}\|qg\|_{L^{2}(\mathcal{L}^2)}\right)\|q\Lambda^{s-1} g\|_{L^{2}(\mathcal{L}^2)} ~~~~~~~~~~~~~~~~~~~~~\\ \notag
		&\lesssim\|u\|_{H^s}\|\nabla_q g\|^2_{H^{s-1}(\mathcal{L}^2)}~.
	\end{align}
	Combining the estimates \eqref{3ineq2}-\eqref{3ineq8} , we infer that
	\begin{align}\label{3ineq9}
		\frac {d} {dt} \|\langle q \rangle \Lambda^{s-1}g\|^2_{L^2(\mathcal{L}^{2})}  &+ 2\|\langle q \rangle \nabla_q \Lambda^{s-1} g\|^2_{L^2(\mathcal{L}^{2})}\\ \notag 
		&\lesssim C\|\nabla_q \Lambda^{s-1} g\|^2_{L^2(\mathcal{L}^{2})} + C_{\delta}\|\nabla \Lambda^{s-1} u\|^2_{L^2}+\delta\|\langle q\rangle\nabla_q\Lambda^{s-1}g\|^2_{L^2(\mathcal{L}^{2})} \\ \notag	 
		&~~~~+\|u\|_{H^s}\left(\|g\|^2_{H^{s}}+\|\langle q\rangle\nabla_qg\|^2_{H^{s-1}(\mathcal{L}^{2})}\right)\\ \notag 
		&\lesssim C\|\nabla_q \Lambda^{s-1} g\|^2_{L^2(\mathcal{L}^{2})} +  C_{\delta}\|\nabla u\|^2_{H^{s-1}}+(\epsilon^{\frac{1}{2}}+\delta)\|\langle q\rangle\nabla_qg\|^2_{H^{s-1}(\mathcal{L}^{2})}~.
	\end{align}
Taking $L^2(\mathcal{L}^2)$ inner product with $\langle q \rangle^2 g $ to system $\eqref{eq1}_3$ , we have
	\begin{align}\label{3ineq10}
		\frac {d} {dt} \|\langle q \rangle g\|^2_{L^2(\mathcal{L}^{2})}  &+ 2\|\langle q \rangle \nabla_q g\|^2_{L^2(\mathcal{L}^{2})} \\ \notag
		&\lesssim C\|\nabla_q g\|^2_{L^2(\mathcal{L}^{2})} + C_{\delta}\|\nabla u\|^2_{L^2}+\delta\|\langle q\rangle\nabla_qg\|^2_{L^2(\mathcal{L}^{2})} \\ \notag
		&~~~+\|u\|_{H^s}\left(\|g\|^2_{L^2}+\|\langle q\rangle\nabla_qg\|^2_{L^2(\mathcal{L}^{2})}\right)\\ \notag
		&\lesssim C\|\nabla_q g\|^2_{L^2(\mathcal{L}^{2})} + C_{\delta}\|\nabla u\|^2_{L^2}+ (\epsilon^{\frac{1}{2}}+\delta)\|\langle q\rangle\nabla_qg\|^2_{L^2(\mathcal{L}^{2})}~.~~~~~~~~~~~~~~~~~~~~~
	\end{align}
	According to \eqref{3ineq9} and \eqref{3ineq10} , we conclude that
	\begin{align}\label{3ineq11}
		\frac {d} {dt} \|\langle q \rangle g\|^2_{H^{s-1}(\mathcal{L}^{2})}  &+ 2\|\langle q \rangle \nabla_q g\|^2_{H^{s-1}(\mathcal{L}^{2})}  \\ \notag 
		&\lesssim C\|\nabla_q g\|^2_{H^{s-1}(\mathcal{L}^{2})} + C_{\delta}\|\nabla u\|^2_{H^{s-1}}+(\epsilon^{\frac{1}{2}}+\delta)\|\langle q\rangle\nabla_qg\|^2_{H^{s-1}(\mathcal{L}^{2})}.~~~~~~~~~~~
	\end{align}
	This complete the proof of Lemma \ref{mix} .
\end{proof}

{\bf The proof of Proposition \ref{pro2} :}  \\
Denote the energy and energy dissipation functionals as follows :
\begin{align*}
	E_{\eta,\lambda}(t)&=\sum_{m=0,s}\left(\|h(\rho)^{\frac 1 2}\Lambda^m\rho \|^2_{L^2}+\|(1+\rho)^{\frac 1 2}\Lambda^m u\|^2_{L^2}+\|\Lambda^mg\|^2_{L^2(\mathcal{L}^{2})} \right) \\ \notag
	& ~~~~ + \sum_{m=0,s-1}\lambda\|\langle q \rangle \Lambda^m g\|^2_{L^2(\mathcal{L}^{2})}+2\eta\sum_{m=0,s-1}\int_{\mathbb{R}^{d}} \Lambda^m u\nabla\Lambda^m \rho dx~,
\end{align*}
and
\begin{align*}
	D_{\eta,\lambda}(t)&=\eta\gamma\|\nabla\rho\|^2_{H^{s-1}}+\mu\|\nabla u\|^2_{H^{s}}+(\mu+\mu')\|{\rm div}u\|^2_{H^{s}}+\|\nabla_q g\|^2_{H^{s}(\mathcal{L}^{2})} \\ \notag
	&~~~+ \lambda\|\langle q \rangle \nabla_q g\|^2_{H^{s-1}(\mathcal{L}^{2})}~.~~~~~~~~~~~~~~~~~~~~~~~~~~~~~
\end{align*}
For some small positive constant $\eta$, we obtain $E_{\lambda}(t)\sim E_{\eta,\lambda}(t)$ and $D_{\lambda}(t)\sim D_{\eta,\lambda}(t)$ . Hence the small assumption $E_{\lambda}(t)\leq\epsilon$ yields that $E_{\eta,\lambda}(t)\lesssim \epsilon$ .
Combining the estimates \eqref{4prop2} , \eqref{4prop3} and \eqref{4prop4} , we arrive at
\begin{align}
	\frac{d}{dt}E_{\eta,\lambda}&(t) + 2D_{\eta,\lambda}(t) \\ \notag
	&\lesssim \left(\epsilon^{\frac{1}{2}} + \left(\frac{\epsilon}{\lambda}\right)^{\frac{1}{2}} +\epsilon + \delta + \eta + \left( \eta + \lambda \right) C_{\delta}\right)\left(\|\nabla \rho\|^2_{H^{s-1}} + \|\nabla u\|^2_{H^{s}} + \|\nabla_q g\|^2_{H^{s}(\mathcal{L}^{2})}\right)~~ \\ \notag
	& ~~~~ + \lambda C\|\nabla_q g\|^2_{H^{s-1}(\mathcal{L}^{2})} +\lambda(\epsilon^{\frac{1}{2}}+\delta)\|\langle q\rangle\nabla_qg\|^2_{H^{s-1}(\mathcal{L}^{2})} \\ \notag
	& \leq D_{\eta,\lambda}~,
\end{align}
where positive constant $\lambda$ , $\delta$ and $\epsilon$ are small enough. Thus we deduce that
\begin{align}
	\sup_{t\in[0,T]} E_{\eta, \lambda}(t) + \int_0^T D_{\eta,\lambda}(t)dt \leq E_{\eta, \lambda}(0)~.
\end{align}
Using the equivalence of $E_{\lambda}(t) \sim E_{\eta,\lambda}(t)$ and $D_{\lambda}(t) \sim D_{\eta,\lambda}(t)$ , there exists constant $C_0>1$ such that
\begin{align}
	\sup_{t\in[0,T]} E_{\lambda}(t) + \int_0^T D_{\lambda}(t)dt \leq C_0E_{\lambda}(0)~.
\end{align}
Therefore, the proof of Proposition \ref{pro2} is completed.
\hfill$\Box$\\

 One can obtain Theorem \ref{th1} from Propositions \ref{pro1} and \ref{pro2} by continuity argument. The proof is omitted here and more details can refer to \cite{2017Global}.
 
  \begin{rema}
 	One may also derive global existence of the system \eqref{eq0} by considering the functional $\mathop{\Sigma}\limits_{n=1}^2 \|\langle q\rangle \nabla_q^n g\|^2_{H^s} + \|\langle q\rangle \nabla_q^3 g\|^2_{H^{s-1}}$ . Nevertheless, the excessive smoothness at micro variable $q\in\mathbb{R}^d$ is technically unnecessary to obtain global priori estimation.
 \end{rema}

 \begin{rema}
 	In comparision to the macro-micro models studied in \cite{2017Global}, it's reasonable for us to remove the term ${\rm div} u \psi$ from the equation $\psi$ obey. Otherwise, one can derive ${\rm div} u =0$ from the equation $\psi$ obey by assumption $\int_{\mathbb{R}^d}\psi dq =\int_{\mathbb{R}^d}\psi_0 dq$. However, the condition $\int_{\mathbb{R}^d}\psi dq =\int_{\mathbb{R}^d}\psi_0 dq$ is of great significance in the estimation of ${\rm div}_q \left(\nabla uq\psi\right)$ and it seems impossible to obtain global priori estimation without any dissipation or conservation law for $\int_{\mathbb{R}^d}\psi dq$. Therefore, we underlined that the system \eqref{eq0} is meaningful and the results in this paper indeed cover those obtained in \cite{2017Global,2018Global} .
 \end{rema}
  
\section{The $L^2$ decay rate}
This section is devoted to investigating the long time behaviour for the micro-macro model for compressible polymeric fluids near equilibrium. The most difficult for us is that the additional stress tensor $\tau$ does not decay fast enough. Thus we failed to use the bootstrap argument as in \cite{Schonbek1985,Luo-Yin,Luo-Yin2}. To deal with this term, we need to use the coupling effect between $\rho$, $u$ and $g$. Different from the classical potential $\mathcal{U}=\frac{1}{2}|q|^2$, the lack of  $\|(\rho,u)\|_{L^1}$  leads the loss of $\frac{d}{8}$ time decay rate, which forces us to obtain weaker $L^2$ decay rate at the very beginning. To over this difficulty, we prove the global solutions of system $(\ref{eq1})$ obtained belong to some Besov space with negative index as that of \cite{luozhaonanFENE}. In reward, it helps to improve the $L^2$ decay rate obtained by using the bootstrap argument. Further, we infer from \cite{He2009} that $\|g\|_{L^2(\mathcal{L}^{2})}$ obtains more $\frac{1}{2}$ decay than $\|(\rho,u)\|_{L^2}$.

To begin with, we rewrite the micro-macro model \eqref{eq1} as follows :
\begin{align}\label{eq3}
	\left\{
	\begin{array}{ll}
		\rho_t + {\rm div}u = F~,  \\[1ex]
		u_t - {\rm div}\Sigma u + \gamma\nabla\rho-{\rm div}\tau = G~,  \\[1ex]
		g_t+\mathcal{L}g-\nabla^i u^j q^i \partial_{q_j}\mathcal{U} + {\rm div}u = H~, \\[1ex]
	\end{array}
	\right.
\end{align}
where $F=-\rm div(\rho u)$, $G=-u\cdot\nabla u+[i(\rho)-1](\rm div\Sigma(u)+\rm div \tau)+[\gamma-h(\rho)]\nabla\rho$ and $H=-u\cdot\nabla g-\frac 1 {\psi_\infty} \nabla_q\cdot(\nabla u qg\psi_\infty)$. To simplify the presentation, we firstly give the following notations.

Denote the energy and energy dissipation functionals considered in this chapter as follows :
\begin{align*}
	&E_0(t)=\|(\rho,u)\|^2_{H^s} + \|g\|^2_{H^{s}(\mathcal{L}^{2})} ~,\\ \notag
	&E_1(t)=\|\Lambda^1(\rho,u)\|^2_{H^{s-1}} + \|\Lambda^1 g\|^2_{H^{s-1}(\mathcal{L}^{2})}~, \\ \notag	&D_0(t)=\gamma\eta\|\nabla\rho\|^2_{H^{s-1}}+\mu\|\nabla u\|^2_{H^{s}}+(\mu+\mu')\|{\rm divu}\|^2_{H^{s}}
	+\|\nabla_q g\|^2_{H^{s}(\mathcal{L}^{2})} ~ ,\\ \notag
	&D_1(t)=\gamma\eta\|\nabla\Lambda^1\rho\|^2_{H^{s-2}}+\mu\|\nabla \Lambda^1u\|^2_{H^{s-1}}+(\mu+\mu')\|{\rm div\Lambda^1u}\|^2_{H^{s-1}}
	+\|\nabla_q \Lambda^1g\|^2_{H^{s-1}(\mathcal{L}^{2})}~.
\end{align*}

Denote the following domains involved in Fourier splitting method  :
$$
S(t) = \{\xi: |\xi|^2 \leq C_d\frac{f'(t)}{f(t)}\} ~~~~ \text{with}~~~~~~ f(t) = 1+t .~~~~~~~~~~~~~~~~~~~~~~~~
$$

Denote the following energy and energy dissipation functionals involved in Fourier splitting method:
\begin{align*}
	~H_0 = \mu\|u\|^2_{H^s} + \eta\gamma \|\rho\|^2_{H^{s-1}}~~~~~~~~\text{and}~~~~~H_1 =  \mu\|\Lambda^1 u\|^2_{H^{s-1}}+\eta\gamma\|\Lambda^1 \rho\|^2_{H^{s-2}}~.
\end{align*}

Denote two important factors $B_1$ and $B_2$ :
$$
B_1 = \int_0^t\|(\rho,u)\|^4_{L^2} ds ~~~~\text{and}~~~~
B_2 = \int_{0}^{t}\int_{S(t)}|\hat{G}\cdot\bar{\hat{u}}| d\xi ds'~.
$$

We give the key lemma of time decay estimation in the following.
\begin{lemm}\label{decaylem1}
	Let $d\geq3$. Assume $(\rho,u,g)$ be a global strong soluition of system \eqref{eq1} with initial data $(\rho_0,u_0,g_0)$ under the conditions in Theorem  \ref{th2} . Then there exist a positive time $T_0$ such that 
	\begin{align}\label{63ineq5}
		\int_{S(t)}\left(|\hat{\rho}|^2+|\hat{u}|^2+\|\hat{g}\|^2_{\mathcal{L}^2}\right)d\xi
		&\lesssim \left(\frac{f'(t)}{f(t)}\right)^{\frac d 2 }\left(1+\|(\rho_0,u_0)\|^2_{\dot{B}^{-\frac d 2}_{2,\infty}}+\|g_0\|^2_{\dot{B}^{-\frac d 2}_{2,\infty}(\mathcal{L}^2)}\right) \\ \notag
		&+ \left(\frac{f'(t)}{f(t)}\right)^{\frac d 2}B_1 + B_2 ~,
	\end{align}
for any $t>T_0$.
\end{lemm}
\begin{proof}
	Applying Fourier transformation to system \eqref{eq3} leads to
	\begin{align}\label{4ineq1}
		~~~~~~~~~~~~
		\left\{
		\begin{array}{ll}
			\hat{\rho}_t+i\xi_{k} \hat{u}^k=\hat{F},  \\[1ex]
			\hat{u}^{j}_t+\mu|\xi|^2 \hat{u}^j+(\mu+\mu')\xi_{j} \xi_{k} \hat{u}^k+i\xi_{j} \gamma\hat{\rho}-i\xi_{k} \hat{\tau}^{jk}=\hat{G}^j,  \\[1ex]
			\hat{g}_t+\mathcal{L}\hat{g}-i\xi_{k} \hat{u}^j R_j \partial_{R_k}\mathcal{U}+i\xi_{k} \hat{u}^k=\hat{H}. \\[1ex]
		\end{array}
		\right.
	\end{align}
	Multiplying $\bar{\hat{\rho}}(t,\xi)$ by system $\eqref{4ineq1}_1$ and taking the real part, we get
	\begin{align}\label{4ineq2}
		\frac 1 2 \frac d {dt} |\hat{\rho}|^2+\mathcal{R}e[i\xi\cdot\hat{u}\bar{\hat{\rho}}]=\mathcal{R}e[\hat{F}\bar{\hat{\rho}}]~.~~~~~~~~~~~~~~~~~~~
	\end{align}
	Multiplying $\bar{\hat{u}}^j(t,\xi)$ by system $\eqref{4ineq1}_2$ and taking the real part, we have
	\begin{align}\label{4ineq3}
		\frac 1 2 \frac d {dt} |\hat{u}|^2+\mathcal{R}e[\gamma\hat{\rho}i\xi\cdot\bar{\hat{u}}]&+\mu|\xi|^2 |\hat{u}|^2+(\mu+\mu')|\xi\cdot\hat{u}|^2\\ \notag
		&=\mathcal{R}e[i\xi\otimes\bar{\hat{u}}(t,\xi):\hat{\tau}]+\mathcal{R}e[\hat{G}\cdot\bar{\hat{u}}].~~~
	\end{align}
	Similar to the above estimate, one can deduce that
	\begin{align}\label{4ineq4}
		\frac 1 2 \frac d {dt} \|\hat{g}\|^2_{\mathcal{L}^2}+
		\|\nabla_q \hat{g}\|^2_{\mathcal{L}^2}=\mathcal{R}e[i\xi\otimes\hat{u}:\bar{\hat{\tau}}]+\mathcal{R}e[\int_{\mathbb{R}^d}\hat{H}\bar{\hat{g}}\psi_\infty dq],
	\end{align}
	Since
	$$\mathcal{R}e[i\xi\cdot\hat{u}\bar{\hat{\rho}}]+\mathcal{R}e[\hat{\rho}i\xi\cdot\bar{\hat{u}}]
	=\mathcal{R}e[i\xi\otimes\bar{\hat{u}}(t,\xi):\hat{\tau}]+\mathcal{R}e[i\xi\otimes\hat{u}:\bar{\hat{\tau}}]=0~,~~~~~$$
it follows that
	\begin{align}\label{4ineq5}
		\frac 1 2 \frac d {dt} (\gamma|\hat{\rho}|^2+|\hat{u}|^2+\|\hat{g}\|^2_{\mathcal{L}^2})&+\mu|\xi|^2 |\hat{u}|^2+(\mu+\mu')|\xi\cdot\hat{u}|^2
		+\|\nabla_q\hat{g}\|^2_{\mathcal{L}^2}  \\ \notag
		&=\mathcal{R}e[\gamma\hat{F}\bar{\hat{\rho}}]+\mathcal{R}e[\hat{G}\cdot\bar{\hat{u}}]+\mathcal{R}e[\int_{\mathbb{R}^d}\hat{H}\bar{\hat{g}}\psi_\infty dq]~.
	\end{align}
	Multiplying $i\xi\cdot\bar{\hat{u}}$ by system $\eqref{4ineq1}_1$ and taking the real part, we get
	\begin{align}\label{4ineq6}
		\mathcal{R}e[\hat{\rho}_t i\xi\cdot\bar{\hat{u}}]-|\xi\cdot\hat{u}|^2=\mathcal{R}e[\hat{F}i\xi\cdot\bar{\hat{u}}]~.~~~~~~~~~~~~~~~~
	\end{align}
	Multiplying $-i\xi_j\bar{\hat{\rho}}$ by system $\eqref{4ineq1}_2$ and taking the real part, we have
	\begin{align}\label{4ineq7}
		\mathcal{R}e[\hat{\rho}i\xi\cdot\bar{\hat{u}}_t]&+\gamma|\xi|^2 |\hat{\rho}|^2+(2\mu+\mu')|\xi|^2 \mathcal{R}e[\hat{\rho}i\xi\cdot\bar{\hat{u }}] \\ \notag &=\mathcal{R}e[\hat{\rho}\xi\otimes\xi:\bar{\hat{\tau}}]+\mathcal{R}e[\bar{\hat{G}}\cdot i\xi\hat{\rho}]~.
	\end{align}
	Combining the estimates \eqref{4ineq2}$-$\eqref{4ineq7}, we obtain
	\begin{align}\label{4ineq8}
		\frac 1 2 &\frac d {dt} \left(\gamma|\hat{\rho}|^2+|\hat{u}|^2+\|\hat{g}\|^2_{\mathcal{L}^2}+2(\mu+\mu')\mathcal{R}e[\hat{\rho}i\xi\cdot\bar{\hat{u}}]\right) \\ \notag
		&~~~~+\mu|\xi|^2 |\hat{u}|^2+(\mu+\mu')\gamma|\xi|^2 |\hat{\rho}|^2
		+\|\nabla_q \hat{g}\|^2_{\mathcal{L}^2}  \\ \notag
		&=-(\mu+\mu')(2\mu+\mu')|\xi|^2 \mathcal{R}e[\hat{\rho}i\xi\cdot\bar{\hat{u }}] +(\mu+\mu')\mathcal{R}e[\hat{\rho}\xi\otimes\xi:\bar{\hat{\tau}}] \\ \notag
		&~~~~+(\mu+\mu')\mathcal{R}e[\hat{F}i\xi\cdot\bar{\hat{u}}] +(\mu+\mu')\mathcal{R}e[\bar{\hat{G}}\cdot i\xi\hat{\rho}]+\mathcal{R}e[\gamma\hat{F}\bar{\hat{\rho}}] \\ \notag
		&~~~~+\mathcal{R}e[\hat{G}\cdot\bar{\hat{u}}]+\mathcal{R}e[\int_{\mathbb{R}^d}\hat{H}\bar{\hat{g}}\psi_\infty dq]~.
	\end{align}
	Owing to $\xi\in S(t)$ such that $|\xi|^2 \leq C_d\frac{f'(t)}{f(t)}$, we infer that
	\begin{align}\label{4ineq9}
		&(\mu+\mu')\mathcal{R}e[\hat{F}i\xi\cdot\bar{\hat{u}}]+(\mu+\mu')\mathcal{R}e[\bar{\hat{G}}\cdot i\xi\hat{\rho}]+\mathcal{R}e[\gamma\hat{F}\bar{\hat{\rho}}] \\
		\notag
		&\leq C\left(|\widehat{\rho u}|^2+|\hat{G}|^2\right)+\frac 1 {10} \left(\mu|\xi|^2 |\hat{u}|^2+(\mu+\mu')\gamma|\xi|^2 |\hat{\rho}|^2\right)~.~~~~~~
	\end{align}
	Taking~ $t$ large enough leads to
	\begin{align}\label{4ineq10}
		2(\mu+\mu')\mathcal{R}e[\hat{\rho}i\xi\cdot\bar{\hat{u}}]\leq
		\frac 1 {10}\left(|\hat{u}|^2+\gamma|\hat{\rho}|^2\right).~~~~~~~~~
	\end{align}
	One can get from Lemmas \ref{Lemma4} and \ref{Lemma5} that
	\begin{align}\label{4ineq11}
		-&(\mu+\mu')(2\mu+\mu')|\xi|^2 \mathcal{R}e[\hat{\rho}i\xi\cdot\bar{\hat{u }}]+(\mu+\mu')\mathcal{R}e[\hat{\rho}\xi\otimes\xi:\bar{\hat{\tau}}] ~~~~~~~~~~~~ \\ \notag
		&\leq\frac 1 {10}\left(\mu|\xi|^2 |\hat{u}|^2+(\mu+\mu')\gamma|\xi|^2 |\hat{\rho}|^2
		+\|\nabla_q \hat{g}\|^2_{\mathcal{L}^2}\right)~,
	\end{align}
and
	\begin{align}\label{4ineq12}
		\begin{split}
			\mathcal{R}e[\int_{B}\hat{H}\bar{\hat{g}}\psi_\infty dR]&\leq C_\delta\int_{B}\psi_\infty|\mathcal{F}(u\cdot\nabla g)|^2
			+\psi_\infty|\mathcal{F}(\nabla u\cdot{R}g)|^2  dR+\delta\|\nabla_q \hat{g}\|^2_{\mathcal{L}^2}~.
		\end{split}
	\end{align}
	Thus combining the estimates \eqref{4ineq9}-\eqref{4ineq12} , we conclude that
	\begin{align}\label{4ineq13}
		|\hat{\rho}|^2+|\hat{u}|^2+\|\hat{g}\|^2_{\mathcal{L}^2}
		&\leq C\left(|\hat{\rho}_0|^2+|\hat{u}_0|^2+\|\hat{g}_0\|^2_{\mathcal{L}^2}\right)+C\int_{0}^{t}|\hat{G}|^2+|\widehat{\rho u}|^2+ |\hat{G}\cdot\bar{\hat{u}}|ds' ~~~ \\ \notag
		&~~~~+C_\delta\int_{0}^{t}\int_{\mathbb{R}^d}\psi_\infty|\mathcal{F}(u\cdot\nabla g)|^2+\psi_\infty|\mathcal{F}(\nabla u\cdot{q}g)|^2 dRds'~.
	\end{align}
	Integrating $\xi$ over $S(t)$ leads to
	\begin{align}\label{4ineq14}
		\int_{S(t)}\left(|\hat{\rho}|^2+|\hat{u}|^2+\|\hat{g}\|^2_{\mathcal{L}^2}\right)d\xi
		&\leq C\int_{S(t)} \left(|\hat{\rho}_0|^2+|\hat{u}_0|^2+\|\hat{g}_0\|^2_{\mathcal{L}^2}\right)d\xi \\ \notag
		& ~~~~+C\int_{S(t)}\int_{0}^{t}\left(|\hat{G}|^2+|\widehat{\rho u}|^2\right)ds'd\xi+ CB_2   \\ \notag
		&~~~~+C_\delta\int_{S(t)}\int_{0}^{t}\int_{\mathbb{R}^d}\psi_\infty|\mathcal{F}(u\cdot\nabla g)|^2 dqds'd\xi \\ \notag
		&~~~~+C_\delta\int_{S(t)}\int_{0}^{t}\int_{\mathbb{R}^d}\psi_\infty|\mathcal{F}(\nabla u\cdot{q}g)|^2 dqds'd\xi.~~~~~~~~~~~~~~~~~~~~
	\end{align}
	It follows from Proposition \ref{prop0} that
	\begin{align}\label{4ineq15}
		\int_{S(t)} \left(|\hat{\rho}_0|^2+|\hat{u}_0|^2+\|\hat{g}_0\|^2_{\mathcal{L}^2}\right)d\xi
		&\leq\sum_{j\leq \log_2[\frac {4} {3}C_d^{\frac 1 2 }\left(\frac{f'(t)}{f(t)}\right)^{-\frac 1 2}]}\left(\|\dot{\Delta}_j (\rho_0,u_0)\|^2_{L^2}+\|\dot{\Delta}_j g_0\|^2_{L^2(\mathcal{L}^2)}\right) \\ \notag
		&\leq\sum_{j\leq \log_2[\frac {4} {3}C_d^{\frac 1 2 }\left(\frac{f'(t)}{f(t)}\right)^{-\frac 1 2}]}2^{jd}\left(\|(\rho_0,u_0)\|^2_{\dot{B}^{-\frac d 2}_{2,\infty}}+\|g_0\|^2_{\dot{B}^{-\frac d 2}_{2,\infty}(\mathcal{L}^2)}\right) \\ \notag
		&\leq C\left(\frac{f'(t)}{f(t)}\right)^{\frac d 2}\left(\|(\rho_0,u_0)\|^2_{\dot{B}^{-\frac d 2}_{2,\infty}}+\|g_0\|^2_{\dot{B}^{-\frac d 2}_{2,\infty}(\mathcal{L}^2)}\right)~.
	\end{align}
	Applying Minkowski's inequality and Theorem \ref{th1} lead to
	\begin{align}\label{4ineq16}
		\int_{S(t)}\int_{0}^{t}|\hat{G}|^2ds'd\xi
		&\leq C\int_{S(t)}d\xi \int_{0}^{t}\||\hat{G}|^2\|_{L^{\infty}}ds' \\ \notag
		&\leq C\left(\frac{f'(t)}{f(t)}\right)^{\frac d 2}\int_0^t\|(\rho,u)\|^2_{L^2}\left(\|\nabla (\rho,u)\|^2_{H^1} + \|\nabla g\|^2_{L^2(\mathcal{L}^2)}\right)ds ~~~~~~~~~~~~~~~ \\ \notag
		&\leq C\left(\frac{f'(t)}{f(t)}\right)^{\frac d 2}~.
	\end{align}
Similar to the above estimates, we obtain
	\begin{align}\label{4ineq17}
		\int_{S(t)}\int_{0}^{t}|\widehat{\rho u}|^2ds'd\xi
		&=\int_{0}^{t}\int_{S(t)}|\widehat{\rho u}|^2 d\xi ds'  \\ \notag
		&\leq C\int_{S(t)}d\xi \int_{0}^{t}\||\widehat{\rho u}|^2\|_{L^{\infty}}ds' \\ \notag
		&\leq C\left(\frac{f'(t)}{f(t)}\right)^{\frac d 2} B_1~.~~~~~~~~~~~~~~~~~~~~~~~~~~~
	\end{align}
	 Lemmas \ref{Lemma4} and Theorem \ref{th1} ensure that
	\begin{align}\label{4ineq18}
		\int_{S(t)}\int_{0}^{t}\int_{\mathbb{R}^d}\psi_\infty|\mathcal{F}(u\cdot\nabla g)|^2 &dqds'd\xi
		+\int_{S(t)}\int_{0}^{t}\int_{\mathbb{R}^d}\psi_\infty|\mathcal{F}(\nabla u\cdot{q}g)|^2 dqds'd\xi \\ \notag
		&\leq C\left(\frac{f'(t)}{f(t)}\right)^{\frac d 2} \int_{0}^{t}\|u\|^2_{L^{2}}\|\nabla g\|^2_{L^{2}(\mathcal{L}^{2})}+\|\nabla u\|^2_{L^{2}}\|\nabla_qg\|^2_{L^{2}(\mathcal{L}^{2})}ds' ~~~~ \\ \notag
		&\leq C\left(\frac{f'(t)}{f(t)}\right)^{\frac d 2}~.
	\end{align}
We thus complete the proof of Lemma \ref{decaylem1} .
\end{proof}

To obtain optimal $L^2$ decay rate, we have the following lemma.
\begin{lemm}\label{decaylem2}
	Let $d\geq3$ and $\alpha\in[\beta,\frac{d}{2}]$. Let the conditions of Theorem  \ref{th2} be fulfilled. If there exists positive time $T_0$ such that 
	$$
	(1+t)^{-\alpha}B_1 + B_2 \leq C(1+t)^{-\beta}~,~~~~
	$$
	for any $t>T_0$. Then it holds that
	\begin{align*}
		E_0(t) \leq C(1+t)^{-\beta}~.
	\end{align*}
\end{lemm}
\begin{proof}
	It follows from Propositions \ref{low} and \ref{high} that
	\begin{align}\label{5ineq1}
		\frac{d}{dt}E_0(t) + D_0 (t) \leq 0~.
	\end{align}
	Consider $S(t) = \{\xi: |\xi|^2 \leq \frac{C_d}{1+t}\}$ . According to Schonbek's strategy shown in \cite{Schonbek1985} , we get
	\begin{align}\label{5ineq2}
		\|\nabla u\|^2_{H^s} &= \int_{S(t)\cup S^c(t)} (1+|\xi|^{2s})|\xi|^2|\hat{u}(\xi)|^2d\xi~ \\ \notag
		& \geq \frac{C_d}{1+t}\int_{S^c(t)} (1+|\xi|^{2s})|\hat{u}(\xi)|^2d\xi~.
	\end{align}
	Similarly, we have
	\begin{align}\label{5ineq3}
		\|\nabla \rho\|^2_{H^{s-1}} \geq \frac{C_d}{1+t}\int_{S^c(t)} (1+|\xi|^{2{s-1}})|\hat{\rho}(\xi)|^2d\xi~,~~
	\end{align}
	from which one can deduce that
	\begin{align}\label{5ineq4}
		\frac{d}{dt}E_0(t) & + \frac{C_d}{1+t}H_0(t) + \|\nabla_qg\|^2_{H^s(\mathcal{L}^2)} \\ \notag
		& \leq \frac{CC_d}{1+t}\int_{S(t)} |\hat{\rho}(\xi)|^2+|\hat{u}(\xi)|^2d\xi~.~~~~~
	\end{align}
	According to Lemma \ref{decaylem1} and condition above, we obtain
	\begin{align}\label{5ineq5}
		\int_{S(t)} |\hat{\rho}(\xi)|^2+|\hat{u}(\xi)|^2d\xi \leq C(1+t)^{-\beta}.
	\end{align}
Consequently,
	\begin{align}\label{5ineq6}
		~~~~~~~~~~~~\frac{d}{dt}E_0(t) + \frac{C_d}{1+t}H_0(t) + \|\nabla_qg\|^2_{H^s(\mathcal{L}^2)} \leq C(1+t)^{-\beta-1}~,
	\end{align}
which implies
	\begin{align}\label{5ineq7}
		\frac d {dt}\left( (1+t)^{\beta+1}E_0(t) \right)&+C_d(1+t)^{\beta}H_0(t)
		+(1+t)^{\beta+1}\|\nabla_q g\|^2_{H^{s}(\mathcal{L}^{2})} \\ \notag
		&\leq CC_d +  (1+t)^{\beta}E_0(t)~.
	\end{align}
Choosing $C_d$ and $T_d$ sufficiently large, one can arrive at
	\begin{align}\label{5ineq8}
		(1+t)^{\beta+1}E_0(t)&\leq C(1+t)+C\int_{0}^{t}\|\Lambda^s \rho\|^2_{L^{2}}(1+s')^{\frac d 2}ds'  \\ \notag
		&\leq C(1+t)+C\int_{0}^{t}D_0(s')(1+s')^{\beta}ds'  \\  \notag
		&\leq C(1+t)+C\int_{0}^{t}E_0(s')(1+s')^{\beta-1}ds'~,~~~~~~~~~~~
	\end{align}
for any $t>T_d$. Define $N(t)=\mathop{\rm sup}\limits_{s\in[0,t]}(1+t)^{\beta}E_0(t)$ , we infer that
	\begin{align}\label{5ineq9}
	~~~~~~~~~~~~	N(t) \leq C + \int_0^t N(s)(1+s)^{-2}ds~,
	\end{align}
which implies
	\begin{align}\label{5ineq10}
		E_0(t) \leq C(1+t)^{-\beta}.~~~~~~
	\end{align}
	The proof of Lemma \ref{decaylem2} is completed.
\end{proof}

By virtue of Fourier spliting method, we obtian the initial decay rate in the following.
\begin{prop}\label{decaypro1}
	Let $d\geq3$. Under the condition of Theorem  \ref{th2} , there exists a constant $C$ such that 
	\begin{align*}
		E_0(t) \leq C(1+t)^{-\frac{d}{4}} ~~~\text{and}~~~E_1 \leq C(1+t)^{-\frac d 4-1}~.
	\end{align*}
for any  $t\in(0,\infty)$.
\end{prop}
\begin{proof}
	We infer from Theroem \ref{th1} that
	\begin{align}\label{6ineq1}
		B_1 &\leq C(1+t)^1~,~~~~~~
	\end{align}
	and
	\begin{align}\label{6ineq2}
		B_2 &\leq C(1+t)^{-\frac{d}{4}} \int_{0}^{t}(\|u\|^2_{L^{2}}+\|\rho\|^2_{L^{2}})D(s')^{\frac 1 2}ds' ~~~\\ \notag
		&\leq C(1+t)^{-\frac{d}{4} + \frac 1 2}~.
	\end{align}
	It follows from Lemma \ref{decaylem2} with $\alpha = \frac{d}{4} + \frac{1}{2}$ and $\beta = \frac{d}{4} - \frac{1}{2}$ that
	\begin{align}\label{6ineq3}
		E_0(t)\leq C(1+t)^{-\frac d 4+\frac 1 2}~.~~~
	\end{align}
	Then we improve the decay rate by estimating $B_1$ and $B_2$ again. According to \eqref{6ineq3} , we have
	$$
	(1+t)^{-\frac{d}{4}-\frac{1}{4}}B_1 + B_2 \leq C(1+t)^{-\frac{d}{4}+\frac{1}{4}} .~~~~~~~~~
	$$
	According to Lemma \ref{decaylem2} with $\alpha = \frac{d}{4} + \frac{1}{4}$ and $\beta = \frac{d}{4} - \frac{1}{4}$, one can arrive at
	\begin{align}\label{6ineq4}
		E_0(t)\leq C(1+t)^{-\frac d 4+\frac{1}{4}}~.~~
	\end{align}
	Reestimating $B_1$ and $B_2$ with $\alpha = \frac{d}{4} + \frac{1}{8}$ and $\beta = \frac{d}{4} - \frac{1}{8}$, we deduce that
	\begin{align}\label{6ineq5}
		E_0(t)\leq C(1+t)^{-\frac d 4+\frac{1}{8}}.~~~
	\end{align}
Reestimating $B_1$ and $B_2$ with $\alpha = \beta = \frac{d}{4}$,  we finally infer that
	\begin{align}\label{6ineq6}
		E_0(t)\leq C(1+t)^{-\frac d 4}~.~~~~~
	\end{align}
	Similar to the proof of Proposition \ref{low} and \ref{high} , we get
	\begin{align}\label{6ineq7}
		\frac d {dt} E_1(t)+D_1 (t)\leq 0~.~~~~~~~~~~~
	\end{align}
	According to Schonbek's strategy and \eqref{6ineq7} , we have
	\begin{align}\label{6ineq8}
		\frac d {dt} E_1(t)+\frac { C_d} {1+t}H_1(t)
		+\|\Lambda^1 \nabla_q g\|^2_{H^{s-1}(\mathcal{L}^{2})} 
		\leq \frac {CC_d} {1+t}\int_{S(t)}|\xi|^2\left(|\hat{u}(\xi)|^2+|\hat{\rho}(\xi)|^2\right) d\xi,
	\end{align}
with
\begin{align}\label{6ineq9}
	\frac {CC_d} {1+t}\int_{S(t)}|\xi|^2\left(|\hat{u}(\xi)|^2+|\hat{\rho}(\xi)|^2\right) d\xi&\leq C{C_d}^2 (1+t)^{-2-\frac{d}{4}}.~~~~~~~
\end{align}
By performing a routine procedure, one can arrive at
	$$
	E_1(t) \leq C(1+t)^{-\frac d 4-1}~.~~~
	$$
	This completes the proof of Proposition \ref{decaypro1} .
\end{proof}
\begin{rema}
The proof of Proposition \ref{decaypro1} ensures the boundness of $B_1$. Thus the main obstacle to obtain optimal decay rate is $B_2$. Without $u\in L^1$, the estimate of $B_2$ shown in \eqref{6ineq2} definitely cause the loss of $\frac{d}{8}$ decay rate. However, the $L^2$ decay rate obtianed in Proposition \ref{decaypro1} guarantees $u \in \dot{B}^{-\frac{d}{2}}_{2,\infty}$ with $d\geq3$. In reward, $L^1 \hookrightarrow \dot{B}^{-\frac{d}{2}}_{2,\infty}$ helps to improve the $L^2$ decay rate obtained initially.
\end{rema}
Next, we will prove the solutions belong to some Besov space with negative index.
\begin{lemm}\label{decaylem3}
	Let $d\geq3$. Under the conditions of Theorem  \ref{th2} , it holds that
	\begin{align*}
		(\rho,u,g)\in L^{\infty}(0,\infty;\dot{B}^{-\frac{d}{2}}_{2,\infty})\times L^{\infty}(0,\infty;\dot{B}^{-\frac{d}{2}}_{2,\infty})\times L^{\infty}(0,\infty;\dot{B}^{-\frac{d}{2}}_{2,\infty}(\mathcal{L}^2)).
	\end{align*}
	\begin{proof}
		Applying $\dot{\Delta}_j$ to system \eqref{eq3}, we infer that
		\begin{align}\label{7ineq1}
			\left\{
			\begin{array}{ll}
				\dot{\Delta}_j\rho_t+{\rm div}\dot{\Delta}_j u=\dot{\Delta}_j F,  \\[1ex]
				\dot{\Delta}_j u_t-{\rm div}\Sigma(\dot{\Delta}_j u)+\gamma\nabla\dot{\Delta}_j \rho-{\rm div}\dot{\Delta}_j\tau=\dot{\Delta}_j G,  \\[1ex]
				\dot{\Delta}_j g_t+\mathcal{L}\dot{\Delta}_j g-\nabla\dot{\Delta}_j u q_j \partial_{q_k}\mathcal{U}+{\rm div} \dot{\Delta}_j u=\dot{\Delta}_j H.~~ \\[1ex]
			\end{array}
			\right.
		\end{align}
		According to $\int_{B} \dot{\Delta}_j g\psi_\infty dR=0$ , we deduce that
		\begin{align}\label{7ineq2}
			&\frac 1 2 \frac d {dt}\left(\gamma\|\dot{\Delta}_j \rho\|^2_{L^2}+\|\dot{\Delta}_j u\|^2_{L^2}+\|\dot{\Delta}_j g\|^2_{L^2(\mathcal{L}^{2})}\right) \\ \notag
			&~~~~+\mu\|\nabla\dot{\Delta}_j u\|^2_{L^2}+(\mu+\mu')\|{\rm div} \dot{\Delta}_j u\|^2_{L^2}+\|\nabla_q \dot{\Delta}_j g\|^2_{L^2(\mathcal{L}^{2})}   \\ \notag
			&=\int_{\mathbb{R}^{2}} \gamma\dot{\Delta}_j F\dot{\Delta}_j \rho dx+\int_{\mathbb{R}^{2}} \dot{\Delta}_j G\dot{\Delta}_j u dx+\int_{\mathbb{R}^{2}}\int_{\mathcal{R}^d} \dot{\Delta}_j H\dot{\Delta}_j g \psi_\infty dxdq   \\ \notag
			&\leq C\left(\|\dot{\Delta}_j F\|_{L^2}\|\dot{\Delta}_j \rho\|_{L^2}+\|\dot{\Delta}_j G\|_{L^2}\|\dot{\Delta}_j u\|_{L^2}\right) \\ \notag
			&~~~~+C\left(\|\dot{\Delta}_j (u\nabla g)\|^2_{L^2(\mathcal{L}^{2})}+\|\dot{\Delta}_j (\nabla uqg)\|^2_{L^2(\mathcal{L}^{2})}\right)+\frac 1 2 \|\nabla_q \dot{\Delta}_j g\|^2_{L^2(\mathcal{L}^{2})} .~~~~~~~~~~~~~~~~~~~~
		\end{align}
		Multiplying $2^{-dj}$ by \eqref{7ineq2} and taking $l^\infty$ norm with respect to $j\in Z$, we obtain
		\begin{align}\label{7ineq3}
			\frac d {dt}\left(\gamma\|\rho\|^2_{\dot{B}^{-\frac{d}{2}}_{2,\infty}}+\|u\|^2_{\dot{B}^{-\frac{d}{2}}_{2,\infty}}+\|g\|^2_{\dot{B}^{-\frac{d}{2}}_{2,\infty}(\mathcal{L}^{2})}\right)  &\lesssim \|F\|_{\dot{B}^{-\frac{d}{2}}_{2,\infty}}\|\rho\|_{\dot{B}^{-\frac{d}{2}}_{2,\infty}}+\|G\|_{\dot{B}^{-\frac{d}{2}}_{2,\infty}}\|u\|_{\dot{B}^{-\frac{d}{2}}_{2,\infty}}\\ \notag
			&~~~~+ \|\nabla uqg\|^2_{\dot{B}^{-\frac{d}{2}}_{2,\infty}(\mathcal{L}^{2})}+\|u\nabla g\|^2_{\dot{B}^{-\frac{d}{2}}_{2,\infty}(\mathcal{L}^{2})}~.
		\end{align}
		Define $M(t)=\mathop{\sum}\limits_{0\leq s'\leq t} \|\rho\|_{\dot{B}^{-\frac{d}{2}}_{2,\infty}}+\|u\|_{\dot{B}^{-\frac{d}{2}}_{2,\infty}}+\|g\|_{\dot{B}^{-\frac{d}{2}}_{2,\infty}(\mathcal{L}^{2})}$ . According to  \eqref{7ineq3}, we get
		\begin{align}\label{7ineq4}
			M^2(t)&\leq M^2(0)+CM(t)\int_0^{t}\|F\|_{\dot{B}^{-\frac{d}{2}}_{2,\infty}}+\|G\|_{\dot{B}^{-\frac{d}{2}}_{2,\infty}}ds'  \\  \notag
			&~~~~+C\int_0^{t}\|\nabla uqg\|^2_{\dot{B}^{-\frac{d}{2}}_{2,\infty}(\mathcal{L}^{2})}+\|u\nabla g\|^2_{\dot{B}^{-\frac{d}{2}}_{2,\infty}(\mathcal{L}^{2})}ds'~.~~~~~~~~~~~~~~
		\end{align}
		Lemmas \ref{Lemma4} , \ref{Lemma5} and Proposition \ref{decaypro1} ensure that the following estimates hold :
		\begin{align}\label{7ineq5}
			\int_0^{t}\|\nabla uqg\|^2_{\dot{B}^{-\frac{d}{2}}_{2,\infty}(\mathcal{L}^{2})}&+\|u\nabla g\|^2_{\dot{B}^{-\frac{d}{2}}_{2,\infty}(\mathcal{L}^{2})}ds'~~~~~~~~~~~~~~~~~~~~~~~~~~~~~~~~~~~~~~~~~~~~~~~~~~~~\\ \notag
			&\leq C\int_0^{t}\|\nabla uqg\|^2_{L^{1}(\mathcal{L}^{2})}+\|u\nabla g\|^2_{L^{1}(\mathcal{L}^{2})}ds'  \\ \notag
			&\leq C\int_0^{t}\|\nabla u\|^2_{L^2}\|\langle q\rangle g\|^2_{L^{2}(\mathcal{L}^{2})}+\|\langle q\rangle\nabla_q\nabla g\|^2_{L^2(\mathcal{L}^{2})}\|u\|^2_{L^{2}}ds' \\ \notag
			&\leq C\int_0^{t}E_0D_0ds'\leq C~,
		\end{align}
		and
		\begin{align}\label{7ineq6}
			\int_0^{t}\|F\|_{\dot{B}^{-\frac{d}{2}}_{2,\infty}}ds'
			&\leq C\int_0^{t}\|F\|_{L^{1}}ds'  \\ \notag
			&\leq C\int_0^{t}\|u\|_{L^2}\|\nabla\rho\|_{L^2}+\|{\rm div}u\|_{L^2}\|\rho\|_{L^2}ds' ~~~~~~~~~~~~~~~~~~~  \\ \notag
			&\leq C\int_0^{t}(1+s')^{-(\frac{1}{2}+\frac{d}{4})}ds'\leq C~.
		\end{align}
		Analogously,
		\begin{align}\label{7ineq7}
			\int_0^{t}\|G\|_{\dot{B}^{-\sigma}_{2,\infty}}ds'
			&\leq C\int_0^{t}(1+s')^{-(\frac{1}{2}+\frac{d}{4})}ds'+C\int_0^{t}\|{\rm div}\tau\|_{L^2}\|\rho\|_{L^2}ds' ~~~ \\ \notag
			&\leq C+C\int_0^{t}\|\rho\|_{L^2}\|\nabla_q\nabla g\|_{L^2(\mathcal{L}^{2})}ds'  \\ \notag
			&\leq C+C\int_0^{t}(1+s')^{-(\frac{1}{2}+\frac{d}{4})}ds' \leq C~.~~~~~~~~~~~~~~~~~~~~~~~~~~~
		\end{align}
		Combining estimates \eqref{7ineq5}-\eqref{7ineq7} , we get $ M^2(t)\leq CM^2(0)+M(t)C+C$, which implies $M(t)\leq C$. This complete the proof of Lemma \ref{decaylem3} .
	\end{proof}
\end{lemm}
\begin{prop}[Large time behaviour]\label{decaythe1}
	Let $d\geq 3$. Let $(\rho,u,g)$ be a global strong solution of system \eqref{eq1} considered in Theorem \ref{th2} . In addition, if $(\rho_0,u_0)\in \dot{B}^{-\frac{d}{2}}_{2,\infty}\times \dot{B}^{-\frac{d}{2}}_{2,\infty}$ and $g_0 \in \dot{B}^{-\frac{d}{2}}_{2,\infty}(\mathcal{L}^2)$, then there exists a constant $C$ such that
	\begin{align*}
		E_0(t) \leq C(1+t)^{-\frac d 2}~~~\text{and}~~~~E_1(t) \leq C(1+t)^{-\frac d 2-1}.
	\end{align*}
\end{prop}
\begin{proof}
	We infer from Proposition \ref{decaypro1} that
	\begin{align}\label{8ineq1}
		(1+t)^{-\frac{d}{2}}B_1 &\leq C(1+t)^{-\frac{d}{2}} \int_{0}^{t}(1+s')^{-\frac{d}{2}}ds'~~~~~~~~~~\\ \notag
		&\leq C(1+t)^{-\frac{d}{2}}.
	\end{align}
	One can get from Proposition \ref{prop0} that
	\begin{align}\label{8ineq2}
		\int_{S(t)}|\hat{u}|^2 d\xi &\leq
		\sum_{j\leq \log_2[\frac {4} {3}C_d^{\frac 1 2 }\left(1+t\right)^{-\frac 1 2}]}\|\dot{\Delta}_j u_0\|^2_{L^2}~~~~~\\ \notag 
		&\leq   \sum_{j\leq \log_2[\frac {4} {3}C_d^{\frac 1 2 }\left(1+t\right)^{-\frac 1 2}]} 2^{d j}M^2(t) ~~~~~~~~~~\\ \notag
		&\leq C(1+t)^{-\frac{d}{2}} , 
	\end{align}
which implies
	\begin{align}\label{8ineq3}
		~~~~~~~~~~~~~~~~~~~~~~~~~~~~B_2
		&\leq C\int_{0}^{t}\|G\|_{L^{1}}\int_{S(t)}|\hat{u}|d\xi ds' \\ \notag
		&\leq C(1+t)^{-\frac d 4}\int_{0}^{t}\|G\|_{L^{1}}\left(\int_{S(t)}|\hat{u}|^2d\xi\right)^{\frac 1 2} ds' ~~~~~~~~~~~~~~~~ \\ \notag
		&\leq C(1+t)^{-\frac {d} {2}}\int_{0}^{t}(1+s')^{-\frac{d}{4}-\frac {1} {2}}ds' \\ \notag
		&\leq C(1+t)^{-\frac{d}{2}}~.
	\end{align}
	Similar to the proof of Proposition \ref{decaypro1} , we have
	\begin{align}\label{8ineq4}
		\frac d {dt} E_0(t)+\frac {C_d} {1+t}H_0(t)
		+\|\nabla_q g\|^2_{H^{s}(\mathcal{L}^{2})}\leq \frac {C} {1+t}(1+t)^{-\frac{d}{2}}~,~~~~
	\end{align}
which implies
	\begin{align}\label{8ineq5}
		E_0(t)\leq C(1+t)^{-\frac{d}{2}}~.~~~~~~~~~
	\end{align}
	According to Schonbek's strategy, we infer that
	\begin{align}\label{8ineq6}
		\frac d {dt} E_1(t)+\frac {C_d} {1+t}H_1(t)
		&+\|\Lambda^1 \nabla_q g\|^2_{H^{s-1}(\mathcal{L}^{2})}  \\ \notag
		&\leq \frac {CC_d} {1+t}\int_{S(t)}|\xi|^2\left(|\hat{u}(\xi)|^2+|\hat{\rho}(\xi)|^2\right) d\xi,
	\end{align}
with
	\begin{align}\label{8ineq7}
	~~~~~	\frac {CC_d} {1+t}\int_{S(t)}|\xi|^2\left(|\hat{u}(\xi)|^2+|\hat{\rho}(\xi)|^2\right) d\xi&\leq C{C_d}^2 (1+t)^{-2}\left(\|\rho\|^2_{L^2}+\|u\|^2_{L^2}\right)~~~~~~~\\ \notag
		&\leq C (1+t)^{-2-\frac{d}{2}},
	\end{align}
which implies $E_1 \leq C(1+t)^{-1-\frac{d}{2}}$ .
We thus complete the proof of Proposition \ref{decaythe1} .
\end{proof}

\section{The $\dot{H}^s$ decay rate}
This section is devoted to the optimal decay rate of $(\rho,u,\tau)$ in $\dot H^s$. To the best of our knowledge, the dissipation of $\rho$ is of great significance to the decay rate of $(\rho,u)$. However, it fails to obtain the decay rate of $(\rho,u)$ in $\dot{H}^s$ by the same way as in $L^2$ since
$$\|\Lambda^s(\rho,u)\|^2_{L^2} + \eta\langle\Lambda^{s-1}u,\Lambda^s\rho\rangle \simeq \|\Lambda^s(\rho,u)\|^2_{L^2}$$
does not hold anymore for any positive $\eta$. Consequently, what we can obtain is
that $\|\Lambda^s(\rho,u)\|^2_{L^2}$ has the same decay rate as that of $\|\Lambda^{s-1}(\rho,u)\|^2_{L^2}$. Motivated by \cite{2020Wuguochun}, however, we aware that the dissipation of $\rho$ only in high frequency fully enables us to obtain optimal decay rate in $\dot{H}^s$. More precisely, to overcome the obstacle above, $\langle\Lambda^{s}\dot{\Delta}_j\rho,\Lambda^{s-1} \dot{\Delta}_ju\rangle$ in high frequency $S^c(t)$ is considered and it results in dissipation $\int_{S^c(t)}|\xi|^{2s}|\widehat{\rho}|^2d\xi$. To make full use of the benefit the dissipation $\rho$ provides, a critical Fourier splitting estimation is established in the following lemma, which helps to obtain the optimal decay rate of $(\rho,u,g)$ in $\dot H^s$.

To begin with, we derive a critical Fourier splitting estimation in the following.
\begin{lemm}\label{decayinH.s}
Let $d\geq 3$. Assume $(\rho,u,g)$ be a global strong solution considered in Theorem \ref{th1} such that $\|(\rho,u,\tau)\|_{H^s} \leq \varepsilon$ with small  positive constant $\varepsilon$. Let $\sigma_R = \{j\in\mathcal{N}|\varphi(2^{-j}\xi)\cap S^c(R)\neq\varnothing\}$. Then for any sufficiently small positive constant $\eta$, it holds that  
	\begin{align}\label{9ineq1}
	&\frac{d}{dt}\left(\|\Lambda^s (\sqrt{\gamma}\rho,u)\|^2_{L^2} + \|\Lambda^s g\|^2_{L^2(\mathcal{L}^2)} +\frac{C_d\eta}{(1+t)\ln^2(e+t)}\mathop{\Sigma}\limits_{j\in\sigma_R}\langle\Lambda^{s}\dot{\Delta}_j\rho,\Lambda^{s-1} \dot{\Delta}_ju\rangle \right)\\ \notag
	& ~~~~+ \frac{C_d\eta\gamma}{(1+t)\ln^2(e+t)}\int_{S^c(R)}|\xi|^{2s}|\widehat{\rho}|^2d\xi + \|\Lambda^{s + 1} u\|^2_{L^2}+ \|\Lambda^{s} g\|^2_{L^2(\mathcal{L}^2)} \\ \notag
	&\lesssim (1+t)^{-\frac{d+2}{4}}\int_{S(t)}|\xi|^{2s}(|\widehat{\rho}|^2 + |\widehat{u}|^2)d\xi + (1+t)^{-\frac{d+2}{4}}\int_{S^c(t)} |\xi|^{2s}|\widehat{\rho}|^2d\xi\\ \notag
	&~~~~+\frac{\eta}{(1+t)^2\ln^2(e+t)}\mathop{\Sigma}\limits_{j\in\sigma_R}\langle\Lambda^{s}\dot{\Delta}_j\rho,\Lambda^{s-1} \dot{\Delta}_ju\rangle + \frac{\eta}{(1+t)\ln^2(e+t)}\mathop{\Sigma}\limits_{j\in\sigma_R}\|\Lambda^{s}\dot{\Delta}_ju\|^2_{L^2},
\end{align}
for any positive constant $R$.
\end{lemm}
\begin{proof}
The proof is divided into two steps. In the first step, the dissipation of $u$ and $g$ is obtained by standard energy estimation. In the second step, we derive the dissipation of $\rho$ in high frequency by virtue of Littlewood-paley decomposition. It what follows, for any $t\in[0,\infty)$, we denote
$$
\rho = \mathcal{F}^{-1}\left(\chi_{S(t)}\hat{\rho}\right) + \mathcal{F}^{-1}\left(\left(1-\chi_{S(t)}\right)\hat{\rho}\right) = \rho^{low} + \rho^{high},
$$
and
$$
u = \mathcal{F}^{-1}\left(\chi_{S(t)}\hat{u}\right) + \mathcal{F}^{-1}\left(\left(1-\chi_{S(t)}\right)\hat{u}\right) = u^{low} + u^{high}.
$$
	\textbf{ Dissipation of $u$ and $g$ : } \\
	Applying $\Lambda^s$ to system $\eqref{eq1}_1$ and taking $L^2$ inner product with $\Lambda^{s}\rho$ , we have
	\begin{align}\label{9ineq2}
		\frac{1}{2}\frac{d}{dt}\|\Lambda^s\rho\|^2_{L^2} + \langle \Lambda^s\rm div~u,\Lambda^s\rho\rangle = \langle \Lambda^sF,\Lambda^s\rho\rangle~.
	\end{align}
	According to Proposition \ref{decaypro1} , we infer that
	$$
	E_0 \leq C(1+t)^{-\frac{d}{2}} ~~~\text{and} ~~~E_1\leq C(1+t)^{-\frac{d}{2}-1}.
	$$
	It follows from Lemma \ref{Lemma3} that
	\begin{align}\label{9ineq3}
		\langle\Lambda^{s}\left({\rm div}u \rho\right),\Lambda^{s}\rho\rangle &\leq \left(\|{\rm div}u\|_{L^{\infty}}\|\Lambda^{s}\rho\|_{L^2} + \|\rho\|_{L^{\infty}}\|\Lambda^{s+1}u\|_{L^2}\right)\|\Lambda^{s}\rho\|_{L^2} \\ \notag
		&\lesssim \varepsilon\|\Lambda^{s+1}u\|^2_{L^2} +  (1+t)^{-\frac{d+2}{4}}\left(\|\Lambda^{s}\rho^{high}\|^2_{L^2}+\|\Lambda^{s}\rho^{low}\|^2_{L^2}\right)~,~~~~~~~~~~~~~~~~~
	\end{align}
	and
	\begin{align}\label{9ineq4}
		\langle\Lambda^{s}\left(u\cdot\nabla\rho\right),\Lambda^{s}\rho\rangle &= \langle u\cdot\nabla\Lambda^{s}\rho,\Lambda^{s}\rho\rangle + \langle[\Lambda^{s},u\cdot\nabla]\rho,\Lambda^{s}\rho\rangle \\ \notag
		&\leq \|{\rm div}u\|_{L^{\infty}}\|\Lambda^{s}\rho\|^2_{L^2} + \left(\|\nabla\rho\|_{L^{d}}\|\Lambda^su\|_{L^{\frac{2d}{d-2}}} + \|\nabla u\|_{\infty}\|\Lambda^s\rho\|_{L^2}\right)\|\Lambda^s\rho\|_{L^2} \\ \notag
		&\lesssim \varepsilon\|\Lambda^{s+1}u\|^2_{L^2} + (1+t)^{-\frac{d+2}{4}}\left(\|\Lambda^{s}\rho^{high}\|^2_{L^2}+\|\Lambda^{s}\rho^{low}\|^2_{L^2}\right)~,
	\end{align}
	which implies
	\begin{align}\label{9ineq5}
		\frac{d}{dt}\|\Lambda^s\rho\|^2_{L^2} &-\langle\Lambda^su,\nabla\Lambda^{s}\rho\rangle \\ \notag
		&\lesssim \varepsilon\|\Lambda^{s+1}u\|^2_{L^2} +  (1+t)^{-\frac{d+2}{4}}\left(\|\Lambda^{s}\rho^{high}\|^2_{L^2}+\|\Lambda^{s}\rho^{low}\|^2_{L^2}\right)~.
	\end{align}
	Applying $\Lambda^s$ to system $\eqref{eq1}_2$ and taking $L^2$ inner product with $\Lambda^{s}u$ , one can arrive at
	\begin{align}\label{9ineq6}
		\frac{1}{2}\frac{d}{dt}\|\Lambda^su\|^2_{L^2} &+\mu\|\nabla\Lambda^s u\|^2_{L^2}+(\mu+\mu')\|{\rm div}\Lambda^su\|^2_{L^2} \\ \notag
		&+ \gamma\langle\Lambda^s\nabla\rho,\Lambda^su\rangle - \langle\Lambda^s\rm div~\tau,\Lambda^su\rangle = \langle \Lambda^sG,\Lambda^su\rangle~,~~~~~~~~~~~~~~~
	\end{align}
	where
	\begin{align}\label{9ineq7}
		\langle\Lambda^{s}G,\Lambda^{s}u\rangle &=- \langle\Lambda^{s}\left(\frac{\rho}{1+\rho}{\rm div}\Sigma u + [h(\rho)-\gamma]\nabla\rho + \frac{\rho}{1+\rho}{\rm div}\tau + u\nabla u\right),\Lambda^{s}u\rangle~.~~~~~
	\end{align}
It follows from $\xi \in S^c(t)$ that
$$\frac{C_d}{1+t}\|\Lambda^{s}u^{high}\|^2_{L^2} \leq \|\Lambda^{s+1}u^{high}\|^2_{L^2}.$$
	According to Lemma \ref{Lemma3} and Proposition \ref{decaypro1} , we obtain
	\begin{align}\label{9ineq8}
		\langle\Lambda^{s}(u\nabla u),\Lambda^{s}u\rangle &\leq \|\Lambda^{s-1}(u\nabla u)\|_{L^2}\|\Lambda^{s+1}u\|_{L^2}\\ \notag
		& \leq \left(\|\Lambda^{s-1}u\|_{L^{\frac{2d}{d-2}}}\|\nabla u\|_{L^{d}} + \|u\|_{L^{d}}\|\Lambda^{s}u\|_{\frac{2d}{d-2}} \right)\|\Lambda^{s+1}u\|_{L^2} \\ \notag
		& \leq \varepsilon\|\Lambda^{s+1}u\|^2_{L^2} + \|u\|^2_{H^s}\|\Lambda^{s}u^{high}\|^2_{L^2} + (1+t)^{-\frac{d+2}{4}}\|\Lambda^{s}u^{low}\|^2_{L^2} \\ \notag
		& \leq \varepsilon\|\Lambda^{s+1}u\|^2_{L^2} + \frac{C}{1+t}\|u\|^{\frac{2}{3}}_{H^s}\|\Lambda^{s}u^{high}\|^2_{L^2} + (1+t)^{-\frac{d+2}{4}}\|\Lambda^{s}u^{low}\|^2_{L^2} \\ \notag
		& \lesssim \varepsilon\|\Lambda^{s+1}u\|^2_{L^2} + \|u\|^{\frac{2}{3}}_{H^s}\|\Lambda^{s+1}u^{high}\|^2_{L^2} + (1+t)^{-\frac{d+2}{4}}\|\Lambda^{s}u^{low}\|^2_{L^2}\\ \notag
		& \lesssim \varepsilon\|\Lambda^{s+1}u\|^2_{L^2} + (1+t)^{-\frac{d+2}{4}}\|\Lambda^{s}u^{low}\|^2_{L^2}~.
	\end{align}
	Similarly, we have
	\begin{align}\label{9ineq9}
		\langle\Lambda^{s}\left(\frac{\rho}{1+\rho} {\rm div}\Sigma u\right),\Lambda^{s}u\rangle &\leq \|\Lambda^{s-1}(\frac{\rho}{1+\rho}{\rm div}\Sigma u)\|_{L^2}\|\Lambda^{s+1}u\|_{L^2} \\ \notag
		& \leq \left(\|\Lambda^{s-1}\rho\|_{L^{\frac{2d}{d-2}}}\|\nabla^2u\|_{L^{d}} + \|\rho\|_{L^{\infty}}\|\Lambda^{s+1}u\|_{L^2}\right)\|\Lambda^{s+1}u\|_{L^2} \\ \notag
		&\lesssim \varepsilon\|\Lambda^{s+1}u\|^2_{L^2} + (1+t)^{-\frac{d+2}{4}}\left(\|\Lambda^{s}\rho^{high}\|^2_{L^2}+\|\Lambda^{s}\rho^{low}\|^2_{L^2}\right),~~~
	\end{align}
	and
	\begin{align}\label{9ineq10}
		\langle\Lambda^{s}([h(\rho)-\gamma]\nabla\rho),\Lambda^{s}u\rangle &\leq \|\Lambda^{s-1}([h(\rho)-\gamma]\nabla\rho)\|_{L^2}\|\Lambda^{s+1}u\|_{L^2}\\ \notag
		& \leq \left(\|\Lambda^{s-1}\rho\|_{L^{\frac{2d}{d-2}}}\|\nabla\rho\|_{L^{d}} + \|\rho\|_{L^{\infty}}\|\Lambda^{s}\rho\|_{L^2}\right)\|\Lambda^{s+1}u\|_{L^2} \\ \notag
		&\lesssim \varepsilon\|\Lambda^{s+1}u\|^2_{L^2} + (1+t)^{-\frac{d+2}{4}}\left(\|\Lambda^{s}\rho^{high}\|^2_{L^2}+\|\Lambda^{s}\rho^{low}\|^2_{L^2}\right).~~~~~
	\end{align}
	Applying Lemmas \ref{Lemma4} and \ref{Lemma5} leads to
	\begin{align}\label{9ineq11}
		\langle\Lambda^{s}&(\frac{\rho}{1+\rho}{\rm div}\tau),\Lambda^{s}u\rangle\\ \notag
		&\leq \|\Lambda^{s-1}(\frac{\rho}{1+\rho}{\rm div}\tau)\|_{L^2}\|\Lambda^{s+1}u\|_{L^2}\\ \notag
		& \leq (\|\Lambda^{s-1}\rho\|_{L^{\frac{2d}{d-2}}}\|\nabla\tau\|_{L^{d}} + \|\rho\|_{L^{\infty}}\|\Lambda^{s}\tau\|_{L^2})\|\Lambda^{s+1}u\|_{L^2} \\ \notag
		&\lesssim \varepsilon\left(\|\Lambda^{s+1}u\|^2_{L^2} + \|\nabla_q\Lambda^sg\|^2_{L^2(\mathcal{L}^2)}\right) + (1+t)^{-\frac{d+2}{4}}\left(\|\Lambda^{s}\rho^{high}\|^2_{L^2}+\|\Lambda^{s}\rho^{low}\|^2_{L^2}\right) .~~
	\end{align}
	Combining the estimates \eqref{9ineq8} - \eqref{9ineq11} , we deduce that
	\begin{align}\label{9ineq12}
		\frac{1}{2}\frac{d}{dt}\|\Lambda^su\|^2_{L^2} &+\mu\|\nabla\Lambda^s u\|^2_{L^2}+(\mu+\mu')\|{\rm div}\Lambda^su\|^2_{L^2} 
		\\ \notag
		&+ \gamma\langle\Lambda^s\nabla\rho,\Lambda^su\rangle - \langle\Lambda^s{\rm div}\tau,\Lambda^su\rangle \\ \notag
		&\lesssim  \varepsilon\left(\|\Lambda^{s+1}u\|^2_{L^2} + \|\nabla_q\Lambda^sg\|^2_{L^2(\mathcal{L}^2)}\right) \\ \notag
		&~~~~+ (1+t)^{-\frac{d+2}{4}}\|\Lambda^{s}(\rho,u)^{low}\|^2_{L^2}~.~~~~~~~~~~~~~~~~~~~~~~
	\end{align}
	Applying $\Lambda^s$ to system $\eqref{eq1}_3$ and taking $L^2(\mathcal{L}^2)$ inner product with $\Lambda^{s}g$ , we get
	\begin{align}\label{9ineq13}
		\frac{d}{dt}\|\Lambda^sg\|^2_{L^2} + \|\nabla_q\Lambda^sg\|^2_{L^2(\mathcal{L}^2)} +\langle\Lambda^su,\Lambda^{s}\rm div~\tau\rangle 
		= \langle\Lambda^sH,\Lambda^{s}g\rangle~,
	\end{align}
where
	\begin{align}\label{9ineq14}
		\langle\Lambda^{s}H,\Lambda^{s}g\rangle &= -\langle\Lambda^{s}\left(u\cdot\nabla g\right),\Lambda^{s}g\rangle+\langle\Lambda^{s}\left(\nabla u qg\right),\nabla_q\Lambda^{s}g\rangle~.~~~~
	\end{align}
	Applying Lemma \ref{Lemma4} leads to
	\begin{align}\label{9ineq15}
		\langle\Lambda^{s}(u\cdot\nabla g),\Lambda^{s}g\rangle &= \langle u\cdot\nabla\Lambda^{s}g,\Lambda^{s}g\rangle + \langle [\Lambda^{s},u\cdot\nabla]g,\Lambda^{s}g\rangle \\ \notag
		&\leq \left(\|{\rm div}u\|_{L^{\infty}} + \|\nabla u\|_{L^{\infty}} \right)\|\Lambda^{s}g\|^2_{L^2(\mathcal{L}^2)} \\ \notag
		&~~~~ + \|\Lambda^{s}u\|_{L^{\frac{2d}{d-2}}}\|\nabla g\|_{L^2(\mathcal{L}^2)}\|\Lambda^{s}g\|_{L^2(\mathcal{L}^2)}  \\ \notag
		& \lesssim \varepsilon\left(\|\Lambda^{s+1}u\|^2_{L^{2}} + \|\nabla_q\Lambda^{s}g\|^2_{L^2(\mathcal{L}^2)}\right),~~~~~~~~~
	\end{align}
and
	\begin{align}\label{9ineq16}
		\langle\Lambda^{s}\left(\nabla u qg\right),\nabla_q\Lambda^{s}g\rangle &\leq \|\Lambda^{s+1}u\|_{L^{2}}\|\langle q\rangle g\|_{L^{\infty}(\mathcal{L}^2)}\|\Lambda^{s}g\|_{L^{2}(\mathcal{L}^2)}\\ \notag
		&~~~~ + \|\nabla u\|_{L^{\infty}}\|q\Lambda^{s}g\|_{L^{2}}\|\Lambda^{s}g\|_{L^{2}(\mathcal{L}^2)}\\ \notag
		& \lesssim \varepsilon\left( \|\Lambda^{s+1}u\|^2_{L^{2}} + \|\nabla_q\Lambda^{s}g\|^2_{L^2(\mathcal{L}^2)}\right)~.~~~~
	\end{align}
	Combining the estimates \eqref{9ineq15} and \eqref{9ineq16} , we have
	\begin{align}\label{9ineq17}
		\frac{d}{dt}\|\Lambda^sg\|^2_{L^2} &+ \|\nabla_q\Lambda^sg\|^2_{L^2(\mathcal{L}^2)} +\langle\Lambda^su,\Lambda^{s}{\rm div}\tau\rangle  \\ \notag
		&\lesssim \varepsilon\left( \|\Lambda^{s+1}u\|^2_{L^{2}} + \|\nabla_q\Lambda^{s}g\|^2_{L^2(\mathcal{L}^2)}\right).~~~~~~~~~~~~~
	\end{align}
	Together with \eqref{9ineq5} , \eqref{9ineq12} and \eqref{9ineq17} , we conclude that
	\begin{align}\label{9ineq18}
		&\frac{d}{dt}\left(\gamma\|\Lambda^s\rho\|^2_{L^2}+\|\Lambda^s u\|^2_{L^2}+\|\Lambda^{s}g\|^2_{L^2(\mathcal{L}^2)}\right) \\ \notag
		&~~~~+\mu\|\nabla\Lambda^s u\|^2_{L^2}+(\mu+\mu')\|{\rm div} \Lambda^su\|^2_{L^2} + \|\nabla_q\Lambda^sg\|^2_{L^2(\mathcal{L}^2)} \\ \notag
		&\lesssim (1+t)^{-\frac{d+2}{4}}\left(\|\Lambda^{s}\rho^{high}\|^2_{L^2} + \|\Lambda^{s}(\rho,u)^{low}\|^2_{L^2}\right).
	\end{align}
	\textbf{ Dissipation of $\rho$ in high frequency : }\\
	Applying $\Lambda^{s}\dot{\Delta}_j$ to system $\eqref{eq1}_1$ and taking $L^2$ inner product with $\Lambda^{s-1}\dot{\Delta}_ju$ , we get
	\begin{align}\label{9ineq19}
		\langle\partial_t \Lambda^{s}\dot{\Delta}_j\rho,\Lambda^{s-1}\dot{\Delta}_ju\rangle - \|\Lambda^{s}\dot{\Delta}_ju\|^2_{L^2} = \langle \Lambda^{s}\dot{\Delta}_jf,\Lambda^{s-1}\dot{\Delta}_ju\rangle~.
	\end{align}
	Accroding to Proposition \ref{decaypro1} , we have
	\begin{align}\label{9ineq20}
		\langle \Lambda^{s}\dot{\Delta}_jF,\Lambda^{s-1}\dot{\Delta}_ju\rangle &=  - \langle \Lambda^{s}\dot{\Delta}_j(\rho u),\Lambda^{s}\dot{\Delta}_ju\rangle \\ \notag
		& \leq \|\Lambda^{s-1}\dot{\Delta}_j(\rho u)\|_{L^2}\|\Lambda^{s+1}\dot{\Delta}_ju\|_{L^2} \\ \notag
		& \leq d_j\left(\|\Lambda^{s-1}\rho\|_{L^\frac{2d}{d-2}}\|u\|_{L^{d}}+\|\rho\|_{L^{d}}\|\Lambda^{s-1}u\|_{L^\frac{2d}{d-2}}\right)\|\Lambda^{s+1}u\|_{L^2}\\ \notag
		&\leq d_j\varepsilon\|\Lambda^{s+1}u\|^2_{L^2}+d_j\|u\|^2_{L^{d}}\|\Lambda^{s}\rho\|^2_{L^2}   \\ \notag
		&~~~~+ d_j\|\rho\|^2_{H^s}\left(\|\Lambda^{s}u^{high}\|^2_{L^2} + \|\Lambda^{s}u^{low}\|^2_{L^2}\right) \\ \notag
		&\lesssim \varepsilon d_j\|\Lambda^{s+1}u\|^2_{L^2} + d_j(1+t)^{-\frac{d}{2}}\left(\|\Lambda^{s}\rho^{high}\|^2_{L^2} + \|\Lambda^{s}(\rho,u)^{low}\|^2_{L^2}\right)~,~~~
	\end{align}
	where $\{d_j\}_{j\in Z}\in l^1$. Thus we obtain
	\begin{align}\label{9ineq21}
		\langle\partial_t \Lambda^{s}&\dot{\Delta}_j\rho,\Lambda^{s-1}\dot{\Delta}_ju\rangle - \|\Lambda^{s}\dot{\Delta}_ju\|^2_{L^2} \\ \notag
		&\lesssim \varepsilon d_j\|\Lambda^{s+1}u\|^2_{L^2} + d_j(1+t)^{-\frac{d}{2}}\left(\|\Lambda^{s}\rho^{high}\|^2_{L^2} + \|\Lambda^{s}(\rho,u)^{low}\|^2_{L^2}\right).~~~~~~~~~~~~~~~~~~~~~~~
	\end{align}
	Applying $\Lambda^{s-1}\dot{\Delta}_j$ to system $\eqref{eq1}_2$ and taking $L^2$ inner product with $\Lambda^{s}\dot{\Delta}_j\rho$ , we arrive at
	\begin{align}\label{9ineq22}
		\langle \Lambda^{s}\dot{\Delta}_j\rho,\partial_t\Lambda^{s-1}\dot{\Delta}_ju\rangle + \gamma\|\Lambda^s\dot{\Delta}_j\rho\|^2_{L^2} &- \langle\Lambda^{s}{\rm div}\Sigma \dot{\Delta}_j u,\Lambda^{s}\dot{\Delta}_j\rho\rangle \\ \notag
		&= \langle\Lambda^{s}\dot{\Delta}_j G,\Lambda^{s}\dot{\Delta}_j\rho\rangle  + \langle\Lambda^{s}{\rm div} \dot{\Delta}_j \tau,\Lambda^{s}\dot{\Delta}_j\rho\rangle~,~~~~~~~~~~~
	\end{align}
	where
	\begin{align}\label{9ineq23}
		\langle\Lambda^{s-1}\dot{\Delta}_j&G,\Lambda^{s}\dot{\Delta}_j\rho\rangle  \\ \notag
		&= -\langle\Lambda^{s-1}\dot{\Delta}_j \left( \frac{\rho}{1+\rho}{\rm div}\Sigma u + [h(\rho)-\gamma]\nabla\rho + \frac{\rho}{1+\rho}{\rm div}\tau + u\nabla u \right) ,\Lambda^{s}\dot{\Delta}_j\rho\rangle.~~~~~
	\end{align}
	Accroding to Proposition \ref{decaypro1} , we have
	\begin{align}\label{9ineq24}
		\langle\Lambda^{s-1}\dot{\Delta}_j(\frac{\rho}{1+\rho} &{\rm div}\Sigma u),\Lambda^{s}\dot{\Delta}_j\rho\rangle \\ \notag
		&\leq d_j\left(\|\Lambda^{s-1}\rho\|_{L^{\frac{2d}{d-2}}}\|\nabla^2 u\|_{L^{d}} + \|\Lambda^{s+1}u\|_{L^{2}}\|\rho\|_{L^{\infty}} \right)\|\Lambda^{s}\rho\|_{L^2}~~~~~~~~~\\ \notag
		&\lesssim \varepsilon d_j \|\Lambda^{s+1}u\|^2_{L^2} + d_j(1+t)^{-\frac{d+2}{4}}\left(\|\Lambda^{s}\rho^{high}\|^2_{L^2} + \|\Lambda^{s}\rho^{low}\|^2_{L^2}\right),~~
	\end{align}
	and
	\begin{align}\label{9ineq25}
		\langle\Lambda^{s-1}\dot{\Delta}_j([h(\rho)-\gamma]&\nabla\rho),\Lambda^{s}\dot{\Delta}_j\rho\rangle \\ \notag
		&\leq d_j\left(\|\Lambda^{s-1}\rho\|_{L^{\frac{2d}{d-2}}}\|\nabla\rho\|_{L^{d}} + \|\Lambda^{s}\rho\|_{L^{2}}\|\rho\|_{L^{\infty}} \right)\|\Lambda^{s}\rho\|_{L^2} ~~~~~~~~~~~~~~~~~ \\ \notag
		&\lesssim d_j(1+t)^{-\frac{d}{4}}\left(\|\Lambda^{s}\rho^{high}\|^2_{L^2} + \|\Lambda^{s}\rho^{low}\|^2_{L^2}\right).
	\end{align}
	Lemma \ref{Lemma3} ensures that
	\begin{align}\label{9ineq26}
		\langle\Lambda^{s-1}\dot{\Delta}_j(\frac{\rho}{1+\rho}&{\rm div}\tau),\Lambda^{s}\dot{\Delta}_j\rho\rangle \\ \notag
		&\leq d_j\left(\|\Lambda^{s-1}\rho\|_{L^{\frac{2d}{d-2}}}\|\nabla\tau\|_{L^{d}} + \|\Lambda^{s}\tau\|_{L^{2}}\|\rho\|_{L^{\infty}} \right)\|\Lambda^{s}\rho\|_{L^2}\\ \notag
		&\lesssim \varepsilon d_j\|\Lambda^{s}g\|^2_{L^2(\mathcal{L}^2)} + d_j(1+t)^{-\frac{d+2}{4}}\left(\|\Lambda^{s}\rho^{high}\|^2_{L^2}+\|\Lambda^{s}\rho^{low}\|^2_{L^2}\right)~,~~~~~~~~~
	\end{align}
	and
	\begin{align}\label{9ineq27}
		\langle\Lambda^{s-1}\dot{\Delta}_j({\rm div}\tau),\Lambda^{s}\dot{\Delta}_j\rho\rangle &+ \langle\Lambda^{s-1}\dot{\Delta}_j({\rm div}\Sigma u),\Lambda^{s}\dot{\Delta}_j\rho\rangle \\ \notag
		&\leq d_j\left(\|\Lambda^{s}g\|^2_{L^2(\mathcal{L}^2)}+\|\Lambda^{s+1}u\|^2_{L^2}\right) + \varepsilon\|\Lambda^{s}\dot\Delta_j\rho\|^2_{L^2}~.~~~~~~~~~~~~~~~~~~
	\end{align}
	Similarly, we have
	\begin{align}\label{9ineq28}
		\langle\Lambda^{s-1}\dot{\Delta}_j(u\nabla u)&,\Lambda^{s}\dot{\Delta}_j\rho\rangle \\ \notag
		&\leq d_j\left(\|\Lambda^{s-1}u\|_{L^{\frac{2d}{d-2}}}\|\nabla u\|_{L^{\infty}} + \|u\|_{L^{d}}\|\Lambda^{s}u\|_{L^{\frac{2d}{d-2}}} \right)\|\Lambda^{s}\rho\|_{L^2}\\ \notag
		&\lesssim
		\varepsilon d_j\|\Lambda^{s+1}u\|^2_{L^2} + d_j(1+t)^{-\frac{d+2}{4}}\left(\|\Lambda^{s}\rho^{high}\|^2_{L^2}+\|\Lambda^{s}(\rho,u)^{low}\|^2_{L^2}\right)~,~~~~~
	\end{align}
	Combining estimates \eqref{9ineq24}-\eqref{9ineq28} , we infer that
	\begin{align}\label{9ineq29}
		\langle \Lambda^{s}\dot{\Delta}_j\rho,\partial_t\Lambda^{s-1}\dot{\Delta}_ju\rangle &+ \gamma\|\Lambda^s\dot{\Delta}_j\rho\|^2_{L^2} \lesssim \varepsilon d_j\left( \|\Lambda^{s+1}u\|^2_{L^2} + \|\Lambda^{s}g\|^2_{L^2(\mathcal{L}^2)}\right) \\ \notag
		&+ d_j(1+t)^{-\frac{d}{4}}\left(\|\Lambda^{s}\rho^{high}\|^2_{L^2}+\|\Lambda^{s}(\rho,u)^{low}\|^2_{L^2}\right)~.~~~~~~~~~~~~~~~~~~~~~~
	\end{align}
	One can get from \eqref{9ineq21} and \eqref{9ineq29} that
	\begin{align}\label{9ineq30}
		\frac{d}{dt}\langle \Lambda^{s}\dot{\Delta}_j\rho,\Lambda^{s-1}\dot{\Delta}_ju\rangle &+ \gamma\|\Lambda^s\dot{\Delta}_j\rho\|^2_{L^2} - \|\Lambda^{s}\dot{\Delta}_ju\|^2_{L^2} \\ \notag
		&\lesssim \varepsilon d_j\left( \|\Lambda^{s+1}u\|^2_{L^2} + \|\Lambda^{s}g\|^2_{L^2(\mathcal{L}^2)}\right) \\ \notag
		&~~~~ + d_j(1+t)^{-\frac{d}{4}}\left(\|\Lambda^{s}\rho^{high}\|^2_{L^2}+\|\Lambda^{s}(\rho,u)^{low}\|^2_{L^2}\right)~,~~~~~~~~~~~~~~~~~~
	\end{align}
	which implies
	\begin{align}\label{9ineq31}
		~~~~~~~~~\frac{d}{dt}&\left(\frac{C_d\eta}{(1+t)\ln^2(e+t)}\langle \Lambda^{s}\dot{\Delta}_j\rho,\Lambda^{s-1}\dot{\Delta}_ju\rangle\right) + \frac{C_d\eta\gamma}{(1+t)\ln^2(e+t)}\|\Lambda^s\dot{\Delta}_j\rho\|^2_{L^2}  \\ \notag
		&\lesssim \eta\varepsilon  d_j\left( \|\Lambda^{s+1}u\|^2_{L^2} + \|\Lambda^{s}g\|^2_{L^2(\mathcal{L}^2)}\right) + \eta d_j(1+t)^{-\frac{d}{4}-1}\left(\|\Lambda^{s}\rho^{high}\|^2_{L^2}+\|\Lambda^{s}(\rho,u)^{low}\|^2_{L^2}\right) \\ \notag
		&~~~~ + \frac{\eta}{(1+t)^2\ln^2(e+t)}\langle \Lambda^{s}\dot{\Delta}_j\rho,\Lambda^{s-1}\dot{\Delta}_ju\rangle +\frac{\eta}{(1+t)\ln^2(e+t)}\|\Lambda^{s}\dot{\Delta}_ju\|^2_{L^2}~.
	\end{align}
	Summing $j\in \sigma_R$ up leads to
	\begin{align}\label{9ineq32}
		\frac{d}{dt}&\left(\frac{C_d\eta}{(1+t)ln^2(e+t)}\mathop{\Sigma}\limits_{j\in\sigma_R}\langle \Lambda^{s}\dot{\Delta}_j\rho,\Lambda^{s-1}\dot{\Delta}_ju\rangle\right) + \frac{C_d\eta\gamma}{(1+t)ln^2(e+t)}\int_{S^c(R)}|\xi|^{2s}|\widehat{\rho}|^2d\xi   \\ \notag
		&\lesssim \eta\varepsilon\left( \|\Lambda^{s+1}u\|^2_{L^2} + \|\Lambda^{s}\tau\|^2_{L^2}\right)
		+ \eta(1+t)^{-\frac{d}{4}-1}\left(\|\Lambda^{s}\rho^{high}\|^2_{L^2}+\|\Lambda^{s}(\rho,u)^{low}\|^2_{L^2}\right) \\ \notag
		&~~~~+\frac{\eta}{(1+t)^2\ln^2(e+t)}\mathop{\Sigma}\limits_{j\in\sigma_R}\langle \Lambda^{s}\dot{\Delta}_j\rho,\Lambda^{s-1}\dot{\Delta}_ju\rangle +\frac{\eta}{(1+t)\ln^2(e+t)}\mathop{\Sigma}\limits_{j\in\sigma_R}\|\Lambda^{s}\dot{\Delta}_ju\|^2_{L^2}~.
	\end{align}
	According to \eqref{9ineq18} and \eqref{9ineq32} , we arrive at \eqref{9ineq1} .
\end{proof}

\begin{prop}[Large time behaviour]\label{decayth2}
		Let $d\geq 3$. Let $(\rho,u,g)$ be a global strong solution of system \eqref{eq1} considered in Theorem \ref{th2} . In addition, if $(\rho_0,u_0)\in \dot{B}^{-\frac{d}{2}}_{2,\infty}\times \dot{B}^{-\frac{d}{2}}_{2,\infty}$ and $g_0 \in \dot{B}^{-\frac{d}{2}}_{2,\infty}(\mathcal{L}^2)$, then there exists a constant $C$ such that
	\begin{align*}
		\|\Lambda^s(\rho,u)\|_{L^2} + \|\Lambda^sg\|_{L^2(\mathcal{L}^2)} \leq C(1+t)^{-\frac d 4 - \frac{s}{2}} .
	\end{align*}
\end{prop}
\begin{proof}
	According to Schonbek's strategy and Lemma \ref{decayinH.s} , we infer that
	\begin{align}\label{10ineq1}
		\frac{d}{dt'}&\left(\|\Lambda^s (\sqrt{\gamma}\rho,u)\|^2_{L^2} + \|\Lambda^s g\|^2_{L^2(\mathcal{L}^2)}+\frac{C_d\eta}{(1+t')\ln^2(e+t')}\mathop{\Sigma}\limits_{j\in\sigma_R}\langle\Lambda^{s}\dot{\Delta}_j\rho,\Lambda^{s-1} \dot{\Delta}_ju\rangle \right)\\ \notag
		&~~~~+ \frac{C_d}{1+t'}\int|\xi|^{2s}|\widehat{u}|^2d\xi + \frac{C_d\eta\gamma}{(1+t')\ln^2(e+t')}\int|\xi|^{2s}|\widehat{\rho}|^2d\xi + \|\Lambda^{s} g\|^2_{L^2(\mathcal{L}^2)} \\ \notag
		&\lesssim \frac{C_d}{1+t'}\int_{S(t')}|\xi|^{2s}|\widehat{u}|^2d\xi + \frac{\eta}{(1+t')\ln^2(e+t')}\int_{S(R)} |\xi|^{2s}(|\widehat{\rho}|^2 + |\widehat{u}|^2)d\xi \\ \notag
		&~~~~+(1+t')^{-\frac{d+2}{4}}\int_{S^c(t')}|\xi|^{2s}|\widehat{\rho}|^2d\xi+\frac{\eta}{(1+t')^2\ln^2(e+t')}\mathop{\Sigma}\limits_{j\in\sigma_R}\langle\Lambda^{s}\dot{\Delta}_j\rho,\Lambda^{s-1} \dot{\Delta}_ju\rangle~.~~~
	\end{align}
	Consider positive time $T_d$ sufficiently large. According to Proposition \ref{decaypro1} , we have
	\begin{align}\label{10ineq2}
		\frac{d}{dt'}&\left(\|\Lambda^s (\sqrt{\gamma}\rho,u)\|^2_{L^2} + \|\Lambda^s g\|^2_{L^2(\mathcal{L}^2)}+\frac{C_d\eta}{(1+t')\ln^2(e+t')}\mathop{\Sigma}\limits_{j\in\sigma_R}\langle\Lambda^{s}\dot{\Delta}_j\rho,\Lambda^{s-1} \dot{\Delta}_ju\rangle \right)\\ \notag
		&~~~~+ \frac{C_d}{1+t'}\int|\xi|^{2s}|\widehat{u}|^2d\xi + \frac{C_d\eta\gamma}{(1+t')\ln^2(e+t')}\int|\xi|^{2s}|\widehat{\rho}|^2d\xi + \|\Lambda^{s} g\|^2_{L^2(\mathcal{L}^2)} \\ \notag
		&\lesssim \frac{C_d}{(1+t')^{s+\frac{d}{2}+1}} + \frac{\eta}{(1+t')^{\frac{d}{2}+1}\ln^2(e+t')(1+R)^{s}}+\frac{\eta(1+R)}{(1+t')^2\ln^2(e+t')}\|\Lambda^{s}u\|^2_{L^2}~,
	\end{align}
for any $t'>T_d$. By performing a routine procedure, one can arrive at
	\begin{align}\label{10ineq3}
		&(1+t)^{s+\frac{d}{2}+2}\left(\|\Lambda^s (\sqrt{\gamma}\rho,u)\|^2_{L^2} + \|\Lambda^sg\|^2_{L^2(\mathcal{L}^2)}+\frac{C\eta}{(1+t)\ln^2(e+t)}\mathop{\Sigma}\limits_{j\in\sigma_R}\langle\Lambda^{s}\dot{\Delta}_j\rho,\Lambda^{s-1} \dot{\Delta}_ju\rangle \right)\\ \notag
		&~~~~~~~\lesssim C_0 + C_d(1+t)^2 +\frac{C\eta(1+t)^{s+2}}{(1+R)^{s}} + \int_0^t \frac{\eta(1+R)}{\ln^2(e+t')}(1+t')^{s+\frac{d}{2}}\|\Lambda^{s}u\|^2_{L^2} dt'~,
	\end{align}
	where $C_0 = C_{\gamma}\left(\|(\rho_0,u_0)\|^2_{H^s} + \|g_0\|^2_{H^s(\mathcal{L}^2)}+\|\langle q\rangle g_0\|^2_{H^{s-1}(\mathcal{L}^2)}\right)$ . It follows from $\xi \in S^c(t)$ that
	$$\frac{C_d}{1+t}\|\Lambda^{s-1}u^{high}\|^2_{L^2} \leq \|\Lambda^{s}u^{high}\|^2_{L^2}.$$
Thus, by considering small positive constant $\eta$, we obtain 
	\begin{align}\label{10ineq4}
		\|\Lambda^s (\sqrt{\gamma}\rho,u)\|^2_{L^2} + \|\Lambda^s g\|^2_{L^2(\mathcal{L}^2)}&+\frac{C_d\eta}{(1+t)ln^2(e+t)}\mathop{\Sigma}\limits_{j\in\sigma_t}\langle\Lambda^{s}\dot{\Delta}_j\rho,\Lambda^{s-1} \dot{\Delta}_ju\rangle \\ \notag
		&\geq \frac{1}{2}\left(\|\Lambda^s (\sqrt{\gamma}\rho,u)\|^2_{L^2} + \|\Lambda^s g\|^2_{L^2(\mathcal{L}^2)}\right)~.~~~~~~~~~~~~~~~~~~
	\end{align}
	Taking $R=t$ leads to
	\begin{align}\label{10ineq5}
		(1+t)^{s+\frac{d}{2}+2}&\left(\|\Lambda^s (\sqrt{\gamma}\rho,u)\|^2_{L^2} + \|\Lambda^s g\|^2_{L^2(\mathcal{L}^2)}\right) \\ \notag
		&\lesssim C_0 + (1+t)^2 + (1+t)\int_0^t\frac{C\eta}{\ln^2(e+t')}(1+t')^{s+\frac{d}{2}}\|\Lambda^{s}u\|^2_{L^2}dt'~.
	\end{align}
	Define $M(t)=\sup_{t'\in[0,t]}(1+t')^{s+\frac{d}{2}}\left(\|\Lambda^s (\sqrt{\gamma}\rho,u)\|^2_{L^2} + \|\Lambda^s g\|^2_{L^2(\mathcal{L}^2)}\right)$ , we obtain
	\begin{align}\label{10ineq6}
		M(t) \leq C_0 + C + C\int_0^t\frac{M(t')}{(1+t')\ln^2(e+t')}dt'~,
	\end{align}
	which implies
	\begin{align}\label{10ineq7}
		\|\Lambda^s (\rho,u)\|^2_{L^2} + \|\Lambda^s g\|^2_{L^2(\mathcal{L}^2)} \leq C(1+t)^{-s-\frac{d}{2}}~.
	\end{align}
	We thus complete the proof of Proposition \ref{decayth2} .
\end{proof}
\textbf{The proof of Theorem \ref{th2}:}\\
	It follows from Lemmas \ref{Lemma4} , \ref{Lemma5} and \eqref{1ineq4} that
	\begin{align}\label{11ineq1}
		\frac {d} {dt} \|g\|^2_{L^2(\mathcal{L}^{2})}+\| g\|^2_{L^2(\mathcal{L}^{2})}\leq C\|\nabla u\|^2_{L^2}.~~~~~~~~~~~~~~~~~~~~~~~~
	\end{align}
	Applying Proposition \ref{decaythe1} leads to
	\begin{align}\label{11ineq2}
		\|g\|^2_{L^2(\mathcal{L}^{2})} &\leq e^{-t}\|g_0\|^2_{L^2(\mathcal{L}^{2})} + C\int_0^t e^{-(t-s)}\|\nabla u\|^2_{L^2}ds \\ \notag
		& \leq e^{-t}\|g_0\|^2_{L^2(\mathcal{L}^{2})} + C\int_0^t e^{-(t-s)}(1+s)^{-\frac{d}{2}-1}ds \\ \notag
		& \leq C(1+t)^{-\frac{d}{2}-1}~.
	\end{align}
Applying $\Lambda^{s-1}$ to system $\eqref{eq1}_3$ and taking $L^2(\mathcal{L}^2)$ inner product with $\Lambda^{s-1}g$ , one can arrive at
	\begin{align}\label{11ineq3}
		~~\frac {d} {dt} \|\Lambda^{s-1} g\|^2_{L^2(\mathcal{L}^{2})}&+\|\Lambda^{s-1} g\|^2_{L^2(\mathcal{L}^{2})} \\ \notag
		&\leq C\|\Lambda^{s-1}u\|^2_{L^2}\|\nabla g\|^2_{H^{s-1}} + 
		C\|qg\|^2_{H^{s-1}}\|\Lambda^su\|^2_{L^2}~.~~~~~~~
	\end{align}
	Theorem \ref{th1} and Proposition \ref{decayth2} guarantee that
	\begin{align}\label{11ineq4}
		~\|\Lambda^{s-1} g\|^2_{L^2(\mathcal{L}^{2})}& \leq e^{-t}\|\Lambda^{s-1} g_0\|^2_{L^2(\mathcal{L}^{2})} +  C\int_0^t e^{-(t-s)}(1+s)^{-\frac{d}{2}-s} ds \\ \notag
		&\leq C(1+t)^{-\frac{d}{2}-s}~.
	\end{align}
	According to \eqref{11ineq2} , \eqref{11ineq4} and Lemma \ref{Lemma3} , we deduce that
	\begin{align}\label{11ineq5}
		\|\Lambda^{\nu} g\|_{L^2(\mathcal{L}^2)} &\leq \|g\|^{1-\frac{\nu}{s-1}}_{L^2(\mathcal{L}^2)}\|\Lambda^{s-1}g\|^{\frac{\nu}{s-1}}_{L^2(\mathcal{L}^2)} \\ \notag
		&\leq C(1+t)^{-\frac{d}{4}-\frac{\nu}{2}-\frac{1}{2}} .~~~~~~~~~~~~~~~
	\end{align}
	Similarly, we conclude that
	\begin{align}\label{11ineq6}
		\|\Lambda^{\gamma} g\|_{L^2(\mathcal{L}^2)} &\leq \|\Lambda^{s-1}g\|^{s-\gamma}_{L^2(\mathcal{L}^2)}\|\Lambda^{s}g\|^{1-s+\gamma}_{L^2(\mathcal{L}^2)} \\ \notag
		&\leq C(1+t)^{-\frac{d}{4}-\frac{s}{2}}~,
	\end{align}
	and
	\begin{align}\label{11ineq7}
		~\|\Lambda^{\sigma}(\rho,u)\|_{L^2(\mathcal{L}^2)} &\leq \|(\rho,u)\|^{1-\frac{\sigma}{s}}_{L^2(\mathcal{L}^2)}\|\Lambda^{s}(\rho,u)\|^{\frac{\sigma}{s}}_{L^2(\mathcal{L}^2)} ~~~~~~\\ \notag
		&\leq C(1+t)^{-\frac{d}{4}-\frac{\sigma}{2}}~.
	\end{align}
 We thus complete the proof of Theorem \ref{th2} .
\hfill$\Box$\\
\smallskip
\noindent\textbf{Acknowledgments} This work was
partially supported by the National Natural Science Foundation of China (No.12171493 and No.11671407), the Macao Science and Technology Development Fund (No. 098/2013/A3), and Guangdong Province of China Special Support Program (No. 8-2015),
and the key project of the Natural Science Foundation of Guangdong province (No. 2016A030311004).


\phantomsection
\addcontentsline{toc}{section}{\refname}
\bibliographystyle{abbrv} 
\bibliography{hookeanref}

\end{document}